\newif\ifHAL
\HALtrue

\ifHAL
\documentclass[11pt,a4paper]{article}
\usepackage[colorlinks,citecolor=cyan,linkcolor=magenta]{hyperref}
\usepackage{fullpage}
\else
\documentclass[numbers,webpdf,imaiai]{ima-authoring-template}%

\usepackage[T1]{fontenc}
\fi

\usepackage[latin1]{inputenc}
\usepackage{amsmath}
\usepackage{amsfonts}
\usepackage{amssymb}
\usepackage{subcaption}
\usepackage{graphicx}
\usepackage{xspace}
\usepackage{tikz}

\usepackage{bm}

\usepackage[colorinlistoftodos]{todonotes}

\usepackage{pict2e}
\usepackage{picture}

\newcommand{\RR}{\mathbb{R}}      

\newcommand{\vertiii}[1]{{\|\kern-0.25ex | #1
		| \kern-0.25ex \|}}

\newcommand{\vece}{{\bm e}}
\newcommand{\mean}[1]{\{\kern-1.1mm\{#1\}\kern-1.1mm\}}                  
\newcommand{\jump}[1]{[\![#1]\!]}                        

\newcommand{\ud}{\,\mathrm{d}}
\newcommand{\ndg}[1]{| \kern -.25mm \|{#1}| \kern -.25mm \|}
\newcommand{\nsdg}[1]{| \kern -.25mm \|{#1}| \kern -.25mm \|_{\rm s}}

\newcommand{\su}{\sum_{K\in \mesh}}

\newcommand{\fes}{\hat{V}_{h}^k}
\newcommand{\fesI}{\hat{V}_{h}^{k,\textsc{(I)}}}
\newcommand{\fesII}{\hat{V}_{h}^{k,\textsc{(II)}}}
\newcommand{\fesE}{\hat{V}_{K}^k}
\newcommand{\upI}{^{\textsc{(I)}}}
\newcommand{\upII}{^{\textsc{(II)}}}

\newcommand{\n}{{\bm n}}
\newcommand{\vtx}{{\bm z}}

\newcommand{\Fall}{\mathcal{F}}

\newcommand{\upi}{^{\mathrm{i}}}
\newcommand{\upb}{^{\mathrm{b}}}

\newcommand{\upd}{^{\mathrm{d}}}
\newcommand{\Fb}{\mathcal{F}\upb}
\newcommand{\Fint}{\mathcal{F}\upi}
\newcommand{\dK}{\partial K}
\newcommand{\dKi}{\partial K\upi}
\newcommand{\dKb}{\partial K\upb}
\newcommand{\FK}{\mathcal{F}_{\dK}}
\newcommand{\Ihk}{\mathcal{\hat{I}}_h^k}
\newcommand{\FKi}{\mathcal{F}_{\dK}\upi}
\newcommand{\FKb}{\mathcal{F}_{\dK}\upb}

\newcommand{\interL}{\mathcal{L}}
\newcommand{\interC}{\mathcal{C}}
\newcommand{\bL}{{\bm L}}

\newcommand{\bn}{{\bm n}}

\newcommand{\bV}{{\bm V}}

\newcommand{\bsigma}{{\bm \sigma}}

\newcommand{\st}{\tq}
\newcommand{\tq}{{\;|\;}}
\newcommand{\DIV}{\nabla{\cdot}}

\newcommand{\eqq}{\mathrel{\mathop:}=}
\newcommand{\SCAL}{{\cdot}}

\newcommand{\calVd}{\mathcal{V}^{(2)}(\Omega)}
\newcommand{\calVq}{\mathcal{V}^{(4)}(\Omega)}
\newcommand{\inod}{i\in\{1{:}d\}}
\newcommand{\uHHO}{^{\textsc{hho}}}

\newcommand{\mesh}{\mathcal{T}}

\newcommand{\ncdg}[1]{| \kern -.25mm \|{#1}| \kern -.25mm \|_{\rm DG}}

\renewcommand{\tilde}[1]{\widetilde{#1}}
\renewcommand{\hat}[1]{\widehat{#1}}
\makeatletter
\newcommand*{\rom}[1]{\text{\expandafter\@slowromancap\romannumeral #1@}}
\makeatother

\graphicspath{{Fig/}}

\ifHAL
\usepackage{theorem}
\newtheorem{corollary}{Corollary}[section]
\newtheorem{lemma}[corollary]{Lemma}
\newtheorem{theorem}[corollary]{Theorem}
\newtheorem{proposition}[corollary]{Proposition}
\newtheorem{definition}[corollary]{Definition}

\newtheorem{remark}[corollary]{Remark}
\newtheorem{example}[corollary]{Example}
\newtheorem{assumption}[corollary]{Assumption}

\newcommand{\qed}{ \vspace{-0.5cm} \hfill $\Box$ }
\newenvironment{proof}[1][Proof.]{\begin{trivlist}
		\item[\hskip \labelsep {\bfseries #1}]}{\end{trivlist}\qed}
\else

\theoremstyle{thmstyletwo}%
\newtheorem{theorem}{Theorem}
\newtheorem{assumption}[theorem]{Assumption}

\newtheorem{lemma}[theorem]{Lemma}
\newtheorem{remark}{Remark}%

\numberwithin{equation}{section}

\fi

\ifHAL

\begin{document}
\else

\begin{document}

\DOI{DOI HERE}
\copyrightyear{2021}
\vol{00}
\pubyear{2021}
\access{Advance Access Publication Date: Day Month Year}
\appnotes{Paper}
\copyrightstatement{Published by Oxford University Press on behalf of the Institute of Mathematics and its Applications. All rights reserved.}
\firstpage{1}


\title[$C^0$-HHO methods for biharmonic problems]{$C^0$-hybrid high-order methods for biharmonic problems}
\fi

\ifHAL
\title{$C^0$-hybrid high-order methods for biharmonic problems}

\author{
	Zhaonan Dong\thanks{
		Inria, 2 rue Simone Iff, 75589 Paris, France,
		and CERMICS, Ecole des Ponts, 77455 Marne-la-Vall\'{e}e, France.
		{\tt{zhaonan.dong@inria.fr}}.
	}
	\and Alexandre Ern\thanks{
		CERMICS, Ecole des Ponts, 77455 Marne-la-Vall\'{e}e, France,
		and  Inria, 2 rue Simone Iff, 75589 Paris, France.
		{\tt{alexandre.ern@enpc.fr}}.
}}
\date{\today}
\else

\author{Zhaonan Dong*
\address{\orgdiv{Inria}, \orgaddress{\street{2 rue Simone Iff}, \postcode{75589}, \state{Paris}, \country{France} } and  \\  \orgdiv{CERMICS}, \orgname{Ecole des Ponts}, \orgaddress{\postcode{77455}, \state{Marne-la-Vall\'{e}e}, \country{France} }}}
\author{Alexandre Ern*
	\address{\orgdiv{CERMICS}, \orgname{Ecole des Ponts},  \orgaddress{\postcode{77455}, \state{Marne-la-Vall\'{e}e}, \country{France} and \\ \orgdiv{Inria}, \orgaddress{\street{2 rue Simone Iff}, \postcode{75589}, \state{Paris}, \country{France} } }}}

\authormark{Author Name et al.}

\corresp[*]{Corresponding author: \href{zhaonan.dong@inria.fr}{zhaonan.dong@inria.fr},
\href{alexandre.ern@enpc.fr}{alexandre.ern@enpc.fr}}

\received{Date}{0}{Year}
\revised{Date}{0}{Year}
\accepted{Date}{0}{Year}


\fi

\ifHAL

\else

\abstract{We devise and analyze $C^0$-conforming hybrid high-order (HHO) methods to approximate biharmonic problems with either clamped or simply supported boundary conditions. $C^0$-conforming HHO methods hinge on cell unknowns which are $C^0$-conforming polynomials of order $(k+2)$ approximating the solution in the mesh cells and on face unknowns which are polynomials of order $k\ge0$ approximating the normal derivative of the solution on the mesh skeleton. Such methods deliver $O(h^{k+1})$ $H^2$-error estimates for smooth solutions. An important novelty in the error analysis is to lower the minimal regularity requirement on the exact solution. The technique to achieve this has a broader applicability than just $C^0$-conforming HHO methods, and to illustrate this point, we outline the error analysis for the well-known $C^0$-conforming interior penalty discontinuous Galerkin (IPDG) methods as well. The present technique does not require a $C^1$-smoother to evaluate the right-hand side in case of rough loads; loads in $W^{-1,q}$, $q>\frac{2d}{d+2}$, are covered, but not in $H^{-2}$.  Finally, numerical results including comparisons to various existing methods showcase the efficiency of the proposed $C^0$-conforming HHO methods.}


\fi

\maketitle

\ifHAL

\abstract{We devise and analyze $C^0$-conforming hybrid high-order (HHO) methods to approximate biharmonic problems with either clamped or simply supported boundary conditions. $C^0$-conforming HHO methods hinge on cell unknowns which are $C^0$-conforming polynomials of order $(k+2)$ approximating the solution in the mesh cells and on face unknowns which are polynomials of order $k\ge0$ approximating the normal derivative of the solution on the mesh skeleton. Such methods deliver $O(h^{k+1})$ $H^2$-error estimates for smooth solutions. An important novelty in the error analysis is to lower the minimal regularity requirement on the exact solution. The technique to achieve this has a broader applicability than just $C^0$-conforming HHO methods, and to illustrate this point, we outline the error analysis for the well-known $C^0$-conforming interior penalty discontinuous Galerkin (IPDG) methods as well. The present technique does not require bubble functions or a $C^1$-smoother to evaluate the right-hand side in case of rough loads.  Finally, numerical results including comparisons to various existing methods showcase the efficiency of the proposed $C^0$-conforming HHO methods.}
\fi

\section{Introduction} \label{Introduction}

Biharmonic PDEs are used in the modelling of various physical phenomena, such as thin plate elasticity, micro-electromechanical systems, and phase separation, to mention a few examples. In the present work, we consider the following model problem with two types of boundary conditions (BC's):
\begin{equation}\label{pde}
	\Delta^2 u =f \quad \text{in $\Omega$}, \qquad
	\left\{\begin{alignedat}{2}
		u&= \partial_{n} u = 0  &\quad &\text{type } (\mbox{I}) \\
		u&= \partial_{nn} u = 0 &\quad &\text{type } (\mbox{II})
	\end{alignedat}\right.\quad \text{on $\partial \Omega$}.
\end{equation}
Here, $\Omega$ is a open bounded Lipschitz domain in $\mathbb{R}^d$, $d\in\{2,3\}$, with boundary $\partial \Omega$, $\partial_n $ denotes the normal derivative on $\partial \Omega$ and $\partial_{nn} $ denotes the normal-normal component of the Hessian  on $\partial \Omega$. In the context of plate modelling, type $(\mbox{I})$ BC is referred to as \emph{clamped} BC, and type $(\mbox{II})$ as \emph{simply supported} BC. Non-homogeneous BC's can be considered, but we focus on homogeneous BC's for simplicity. Instead, dealing with the boundary condition $\partial_{nnn}u=0$ requires further developments. The regularity of the source term $f$ is specified below in Assumption~\ref{ass:regularity}. We also observe that the present developments hinge on the weak formulation of \eqref{pde} involving the Hessian.

The goal of the present work is twofold. The first main goal is to devise and analyze a novel $C^0$-conforming approximation method for the above model problem. The proposed method belongs to the class of hybrid high-order (HHO) methods. These methods were introduced in \cite{DiPEL:14} for linear diffusion and in \cite{DiPEr:15} for locking-free linear elasticity. Since then, they have undergone a vigorous development, as reflected, e.g., in the two recent monographs \cite{di2020hybrid,cicuttin2021hybrid}. Moreover, as discussed in \cite{CoDPE:16,Bridging_HHO,HHOHMM}, HHO methods are closely related to hybridizable discontinuous Galerkin (HDG) methods, weak Galerkin (WG) methods, nonconforming virtual element methods (ncVEM), and multiscale hybrid-mixed (MHM) methods. Interestingly, the present $C^0$-conforming HHO method ($C^0$-HHO in short) can be viewed as a simple approach to extend $C^0$-finite element methods for biharmonic problems by simply adding an additional unknown attached to the mesh faces representing the normal derivative of the solution. $C^0$-HHO methods for the biharmonic problem have not yet been explored in the literature; we refer the reader to \cite{BoDPGK:18,DongErn2021biharmonic,DongErn2021singular} for fully discontinuous HHO methods. The starting point for devising the present $C^0$-HHO methods are the HHO methods from \cite{DongErn2021biharmonic}.

$C^0$-conforming approximation methods are popular to discretize biharmonic problems since such methods avoid the intricate construction of $C^1$-conforming approximation spaces while also avoiding severe conditioning issues that can arise with fully nonconforming approximation methods. Examples of $C^0$-conforming approximation methods from the literature include the classical Morley \cite{morley1968,WangXu:06} and Hsieh--Clough--Tocher (HCT) finite element methods, the $C^0$-interior penalty discontinuous Galerkin (IPDG) method \cite{EGHLMT:02,BrennerC0}, the $C^0$-weak Galerkin (WG) method \cite{MuWYZ:14,CheFe:16}, and the $C^0$-virtual element method (VEM) \cite{ZhChZ:16}. The Morley FEM and the HCT FEM for $d=3$ are lowest-order methods, whereas the HCT FEM for $d=2$ \cite{HCT} and the $C^0$-\{IPDG,WG,HHO,VEM\} for $d\ge2$ can reach arbitrary approximation order. $C^0$-IPDG attaches discrete unknowns to the mesh cells only, whereas $C^0$-\{WG,HHO\} attach discrete unknowns to the mesh cells and faces. A numerical comparison between $C^0$-IPDG and $C^0$-HHO is included herein, indicating the computational efficiency of the $C^0$-HHO approach. Concerning $C^0$-WG, we recall that the only relevant difference between HHO and WG lies in the choice of the discrete unknowns and the design of the stabilization operator. In general, WG employs plain least-squares stabilization, leading to suboptimal convergence rates, whereas HHO employs a more elaborate form of stabilization leading to optimal convergence rates. In other words, to achieve the same convergence rate, $C^0$-WG with plain least-squares stabilization requires more discrete unknowns than $C^0$-HHO methods. Furthermore, among the above $C^0$-conforming methods, $C^0$-VEM is currently the only one supporting general (polytopal) meshes. We notice that an interesting perspective to the present work in order to devise a $C^0$-HHO method on general meshes is to use $C^0$-VEM to build the cell basis functions, while keeping the current face basis functions to handle the normal derivative; this perspective is left to future work.
Finally, $C^0$-conforming HDG methods have been so far seldom considered in the literature, with the exception of a brief discussion in \cite{C0HDG}. Therefore, leveraging on \cite{CoDPE:16}, the present work can be viewed as a contribution to the development of $C^0$-HDG methods for the biharmonic problem.

The second main goal of the present work is to improve on the regularity requirement on the exact solution to lead the error analysis. The technique to achieve this is a non-trivial extension of ideas from \cite{ErnGuer2021} originally developed in the context of second-order elliptic problems (see also \cite[Chap. 40 \& 41]{Ern_Guermond_FEs_II_2021}). The result derived herein can be applied to the $C^0$-HHO method when approximating the biharmonic problem, but also, more broadly, to the other $C^0$-conforming approximation methods discussed above. 
Recall that the difficulty comes from the lack of $C^1$-conformity which causes some difficulties in the error analysis when it comes to bounding the consistency error. A first possibility is to use bubble functions together with a $C^1$-smoother when evaluating the right-hand side of the discrete problem, as shown in \cite{VZ2,VZ3} for the Morley element and the $C^0$-IPDG method (see also \cite{carstensen2021lowerorder} for further results concerning lowest-order methods). However, these techniques so far meet with difficulties when it comes to devising a $C^1$-smoother of arbitrary order on tetrahedral meshes. The alternative road, which is the one followed by most of the above works and also herein, is to require some (mild) additional regularity assumption on the source term beyond $H^{-2}(\Omega)$ and on the exact solution beyond $H^2(\Omega)$. The present work hinges on the following original assumption.
\begin{assumption}[Regularity]\label{ass:regularity}
	We assume that there are real numbers $p>2$ and $q\in (\frac{2d}{d+2},2]$ such that
	$u\in W^{2,p}(\Omega)$, $\Delta u\in W^{1,q}(\Omega)$, and $f\in W^{-1,q}(\Omega).$
\end{assumption}
Assumption~\ref{ass:regularity} is, to our knowledge, novel in the analysis of biharmonic problems. If one prefers to remain in the Hilbertian setting, a simpler, but less general, assumption is $u\in H^{2+s}(\Omega)$, with $s>0$, $\Delta u \in H^1(\Omega)$,
and $f\in H^{-1}(\Omega)$. Notice that the assumption $u\in H^{2+s}(\Omega)$, $s>0$, is still tighter than the one usually made in the literature to analyze $C^0$-conforming methods, which is $u\in H^{2+s}(\Omega)$, $s>\frac12$; the assumption $s>0$ is considered in \cite{carstensen2021lowerorder}.
Finally, we observe that in Assumption~\ref{ass:regularity}, the Laplacian of the solution is smoother than its Hessian. Since the Laplace operator contains a relatively large kernel, this assumption is reasonable.

The rest of this work is organized as follows. We present the key identities to bound the consistency error in Section \ref{sec: continuous_setting}. The results presented in this section have a wider outreach beyond $C^0$-HHO methods since they can be applied to other $C^0$-conforming methods (but not to fully nonconforming methods). Then, we devise $C^0$-HHO methods for both  types of BC's and establish stability and well-posedness of the discrete problems in Section \ref{sec:discrete_setting}.
In Section \ref{sec:error_analysis}, we perform the error analysis under Assumption \ref{ass:regularity}. For completeness, we also outline how the $C^0$-IPDG method can be analyzed under this assumption. Finally, we discuss our numerical results in Section \ref{sec: numerical examples}.

\section{Key identities to bound the consistency error}\label{sec: continuous_setting}

In this section, we introduce some basic notation, present the weak formulation for both
types of BC's, and derive the key identities to bound the consistency error. The main
result of this section, Lemma~\ref{Lemma:identity_II}, can be applied to many
$C^0$-conforming methods.

\subsection{Basic notation} \label{sec:notation}

We use standard notation for the Lebesgue and Sobolev spaces. In particular, for the fractional-order Sobolev spaces, we consider the Sobolev--Slobodeckij seminorm based on the double integral. For an open, bounded, Lipschitz set $S$ in $\mathbb{R}^d$, $d\in\{1,2,3\}$, we denote by $(v,w)_S$ the $L^2(S)$-inner product with appropriate Lebesgue measure, and we employ the same notation for vector- or matrix-valued fields (such fields are denoted using boldface notation). We denote by $\nabla w$ the (weak) gradient of $w$ and by $\nabla^2 w$ its (weak) Hessian. Let $\n_S$ be the unit outward normal vector on the boundary $\partial S$ of $S$. Assuming that the functions $v$ and $w$ are smooth enough, we have the following integration by parts formula:
\begin{equation} \label{eq:ipp1}
	(\Delta^2 v,w)_S = (\nabla^2v,\nabla^2w)_S+(\nabla\Delta v,\n w)_{\partial S}-(\nabla^2v\n,\nabla w)_{\partial S}.
\end{equation}
Whenever the context is unambiguous, we denote by $\partial_n$ the (scalar-valued) normal derivative on $\partial S$ and by $\partial_t$ the ($\RR^{d-1}$-valued) tangential derivative.
We also denote by $\partial_{nn} v$ the (scalar-valued) normal-normal second-order derivative, and by $\partial_{nt} v$ the ($\RR^{d-1}$-valued) normal-tangential second-order derivative. The integration by parts formula~\eqref{eq:ipp1} can then be rewritten as
\begin{equation} \label{eq:ipp2}
	(\Delta^2 v,w)_S = (\nabla^2v,\nabla^2w)_S+(\partial_n\Delta v,w)_{\partial S}-(\partial_{nn}v,\partial_n w)_{\partial S}-(\partial_{nt}v,\partial_t w)_{\partial S}.
\end{equation}
In what follows, the set $S$ is always a polytope so that its boundary can be decomposed into a finite union of planar faces with disjoint interiors. Expressions involving the tangential derivative on $\partial S$ are then implicitly understood to be evaluated as a summation over the faces composing $\partial S$.

Let $\{\mesh\}_{h>0}$ be a shape-regular family of simplicial meshes such that each mesh covers the domain $\Omega$ exactly. A generic mesh cell is denoted by $K\in\mesh$, its diameter by $h_K$, and its unit outward normal by $\n_K$.
We partition the boundary $\dK$ of any mesh cell $K\in\mesh$ by means of the two subsets $\dKi :=  \overline{\dK\cap \Omega}$ and $\dKb:=\dK\cap \partial\Omega$. The mesh faces are collected in the set $\Fall$, which is split as $\Fall=\Fint\cup\Fb$, where $\Fint$ is the collection of the interior faces (shared by two distinct mesh cells) and $\Fb$ the collection of the boundary faces. We orient every mesh interface $F\in\Fint$ by means of the fixed unit normal vector $\n_F$ whose direction is arbitrary but fixed once and for all, whereas we orient every mesh boundary face $F\in\Fb$ by means of the vector $\n_F:=\n_\Omega$. For any mesh cell $K\in\mesh$, the mesh faces composing its boundary $\partial K$ are collected in the set $\FK$, which is partitioned as $\FK=\FKi\cup \FKb$ with obvious notation.

For any real number $\theta\ge0$, we consider the broken Sobolev space
\begin{equation}
	H^\theta(\mesh): = \{v\in L^2(\Omega)\, |\, v_K:= {v}_{|K} \in H^\theta(K), \forall K\in \mesh \}.
\end{equation}
The jump and average of any function $v \in  H^{\theta}(\mesh)$, $\theta>\frac12$, across any mesh interface $F =  \dK_1 \cap  \dK_2\in \Fint$ are defined by setting $\jump{v}_F(x):= v_{|K_1}(x) - v_{|K_2}(x) $ and $\mean{v}_F(x):=  \frac12 (v_{|K_1}(x) + v_{|K_2}(x)) $ for a.e.~$x\in F$, respectively, where $K_1$ is such that its outward unit normal is $\n_F$.

\subsection{Weak formulations}

Recall from Assumption~\ref{ass:regularity} that $f\in W^{-1,q}(\Omega)$
with $q\in (\frac{2d}{d+2},2]$. Let $q'\in [2,\frac{2d}{d-2})$
be such that $\frac{1}{q}+\frac{1}{q'}=1$.
The weak formulation of the biharmonic problem
with type (I) BC's is as follows: Find  $u\upI\in H^2_0(\Omega)$ such that
\begin{equation}\label{weak_form_I}
	(\nabla^2 u\upI, \nabla^2 v )_\Omega = \ell(v) := \langle f,v \rangle_{W^{-1,q},W^{1,q'}_0}, \qquad  \forall v \in H^2_0(\Omega).
\end{equation}
Notice that the right-hand side is meaningful since the Sobolev embedding theorem implies that $H^2(\Omega)\hookrightarrow W^{1,q'}(\Omega)$. Moreover, the model problem~\eqref{weak_form_I} is well-posed owing to the Lax--Milgram lemma.

The weak formulation of the biharmonic problem with type (II) BC's is as follows:
Find $u\upII\in H^1_0(\Omega)\cap H^2(\Omega)$ such that
\begin{equation}\label{weak_form_II}
	(\nabla^2 u\upII, \nabla^2 v )_\Omega = \ell(v), \qquad  \forall v \in H^1_0(\Omega)\cap H^2(\Omega).
\end{equation}
Owing to the Lax--Milgram lemma, this problem is well-posed (recall that the $H^2$-seminorm defines a norm on $H^1_0(\Omega)\cap H^2(\Omega)$). We also observe that the weak formulations \eqref{weak_form_I} and \eqref{weak_form_II} employ the same bilinear form and the same right-hand side; only the trial and test spaces differ.

\subsection{Key identity under the classical regularity assumption}

Before stating our main result based on our new regularity assumption (Assumption~\ref{ass:regularity}), it is useful to illustrate the main idea under the classical regularity assumption $f\in H^{-2+s}(\Omega)$ and $u\in H^{2+s}(\Omega)$ with $s\in (\frac12,1]$.

\begin{lemma}[Key identity]\label{Lemma:identity_I}
	Let $s\in (\frac{1}{2}, 1]$. The following holds for all $v\in H^{2+s}(\Omega)$ and
	all $w\in H^2(\mesh) \cap H_0^{2-s}(\Omega)$:
	\begin{equation} \label{eq:identity_I}
		\langle \Delta^2 v,w \rangle_{H^{-2+s},H_0^{2-s}}
		= \su \Big\{ (\nabla^2 v, \nabla^2 {w}_K)_K
		- (\partial_{nn}  v , \partial_{n} w_K)_{\dK}\Big\}.
	\end{equation}
\end{lemma}
\begin{proof}
	The proof hinges on a density argument. We consider a mollified sequence of functions $v_\delta\in C^{\infty}(\overline{\Omega})$ such that $\lim_{\delta \rightarrow 0} v_\delta = v$ in $H^{2+s}(\Omega)$. We apply \eqref{eq:ipp2} to $v_{\delta|K}$ and $w_K:=w_{|K}$ for all $K\in\mesh$, and sum the result cellwise. This gives
	\[
	(\Delta^2 v_\delta,w)_\Omega = \su \Big\{ (\nabla^2 v_\delta, \nabla^2 {w}_K)_K
	+ (\partial_n\Delta v_\delta,w_K)_{\dK}
	- (\partial_{nn}  v_\delta , \partial_{n} w_K)_{\dK}
	- (\partial_{nt}  v_\delta , \partial_{t} w_K)_{\dK} \Big\}.
	\]
	The left-hand side is equal to $\langle \Delta^2 v_\delta,w \rangle_{H^{-2+s},H_0^{2-s}}$.
	Moreover, the second and fourth summations on the right-hand side vanish. For the second summation, this follows from the fact that $\partial_n\Delta v_\delta$ is single-valued at the mesh interfaces, whereas $w$ is single-valued at the mesh interfaces and vanishes at the mesh boundary faces. Similar arguments are invoked for the fourth summation. Altogether, we obtain
	\[
	\langle \Delta^2 v_\delta,w \rangle_{H^{-2+s},H_0^{2-s}}
	= \su \Big\{ (\nabla^2 v_\delta, \nabla^2 {w}_K)_K
	- (\partial_{nn}  v_\delta , \partial_{n} w_K)_{\dK} \Big\}.
	\]
	We can now pass to the limit $\delta\to0$. This is straightforward for the right-hand side,
	whereas for the left-hand side, we use \cite[Theorem 1.4.4.6]{Grisvard} which gives $\lim_{\delta \rightarrow 0} \Delta^2 v_{\delta} \rightarrow \Delta^2 v $ in $H^{-2+s}(\Omega)$. This proves
	the identity \eqref{eq:identity_I}.
\end{proof}

\begin{remark}[Literature]
	The identity from Lemma~\ref{Lemma:identity_I} is classical; see, e.g.,
	\cite{BrennerC0}. The present proof is, however, different. The advantages are
	that it avoids the decomposition of the solution into singular and regular parts and
	that it works seamlessly in any space dimension.
\end{remark}

\subsection{Key identity under the new regularity assumption}\label{sec:identity_II} 

Extending the identity from Lemma~\ref{Lemma:identity_I} to the more general
setting of Assumption~\ref{ass:regularity} requires giving a meaning to the trace of second-order derivatives on each mesh face individually. Before doing this, we briefly recall some material from \cite{ErnGuer2021} which was originally devised in the context of second-order elliptic problems.

Let $p$ and $q$ be two real numbers as in Assumption~\ref{ass:regularity}, i.e., $p>2$ and $q\in (\frac{2d}{d+2},2]$. Let $\varrho\in (2,p]$ be such that $q\ge \frac{\varrho d}{\varrho + d}$
(this is indeed possible since the function
$x\mapsto \frac{xd}{x+d}$ is increasing on $[2,\infty)$) and let $\varrho' \in [1,2)$
be such that $\frac{1}{\varrho}+\frac{1}{\varrho'}=1$. It is shown in \cite{ErnGuer2021} that for every mesh cell $K\in\mesh$, it is possible to give a meaning separately on every face $F\in\FK$ to the normal component of fields in
\begin{equation}
\bV\upd(K)\eqq \{\bsigma \in \bL^p(K)\st \DIV \bsigma\in L^q(K)\}.
\end{equation}
Specifically, one defines the operator $\gamma\upd_{K,F}:\bV\upd(K) \to
(W^{\frac{1}{\varrho},\varrho'}(F))'$ such that for all $\phi\in W^{\frac{1}{\varrho},\varrho'}(F)$,
\begin{equation} \label{Def: trace operator}
\langle \gamma_{K,F}\upd(\bsigma), {\phi} \rangle_F: = \int_K\Big( \bsigma \SCAL \nabla L_F^K(\phi) +(\nabla \SCAL \bsigma) L_F^K(\phi)  \Big)  \ud x,
\end{equation}
where  $L_F^K:W^{\frac{1}{\varrho},\varrho'}(F) \rightarrow W^{1,\varrho'}(K)$ is a face-to-cell lifting operator satisfying $L_K^F(\phi)_{|F} = \phi$ and $L_K^F(\phi)_{|\FK \backslash F} = 0$
(see \cite[Lemma 3.1]{ErnGuer2021}). Notice that $\gamma\upd_{K,F}(\bsigma)
= (\bsigma\SCAL\bn_K)_{|F}$ whenever the field $\bsigma$ is smooth.
Consider now the functional space
\begin{equation}\label{eq:calVd}
\calVd \eqq \{v\in W^{1,p}(\Omega)\st  \Delta v \in  L^q(\Omega)\}.
\end{equation}
(The superscript refers to the context of second-order PDEs.)
We have $(\nabla v)_{|K}\in \bV\upd(K)$
for all $v\in \calVd$ and all $K\in\mesh$. Moreover, since $\rho>2$, we have $(w_K)_{|F}\in W^{\frac{1}{\varrho},\varrho'}(F)$ for all $w\in H^1(\mesh)$, all $K\in\mesh$, and all $F\in\FK$. Therefore, it is meaningful to define the
following bilinear form for all $v\in \calVd$ and all $w\in H^1(\mesh)$:
\begin{align} \label{eq:def_n_sharp_2}
n_\sharp^{(2)}(v,w):= \su \Big\{ \sum_{F\in \FKi} \langle \gamma_{K,F}\upd( \nabla v), w_{K|F} - \mean{w}_{F}\rangle_F + \sum_{F\in \FKb} \langle \gamma_{K,F}\upd( \nabla v), w_{K|F}\rangle_F \Big\}.
\end{align}
Since $w_{K|F} - \mean{w}_{F} = \frac12(\n_F {\cdot}\n_K)\jump{w_h}_F$ for all $F\in\FKi$, we observe that the bilinear form defined in~\eqref{eq:def_n_sharp_2} coincides with the one defined in \cite[Equ.~(3.12)]{ErnGuer2021} upon setting therein $\theta_{K,F}=\frac12$ if $F\in\FKi$ and $\theta_{K,F}=1$ otherwise. Hence, invoking \cite[Lemma~3.3]{ErnGuer2021}, we infer that for all $v\in \calVd$ and all $w\in H^1(\mesh)$, the following important relation holds:
\begin{align}\label{eq:identity_2}
-(\Delta v, w)_{\Omega} = \su (\nabla v,\nabla w_K)_{K} - n_\sharp^{(2)}(v, w).
\end{align}

We are now ready to address the biharmonic problem. We define the functional space
\begin{equation}\label{eq:calVq}
\calVq \eqq \{v\in W^{2,p}(\Omega)\st \Delta v \in  W^{1,q}(\Omega)\}.
\end{equation}
(The superscript refers to the context of fourth-order PDEs.)
We also define the following bilinear form for all $v\in\calVq$ and all $w\in H^2(\mesh)$:
\begin{equation} \label{eq:def_n_sharp}
n_\sharp^{(4)}(v,w) := \sum_{\inod} n_\sharp^{(2)}(\partial_i v,\partial_i w).
\end{equation}
This definition is meaningful since for all $v\in\calVq$ and all $w\in H^2(\mesh)$,
we have $\partial_i v\in \calVd$ and $\partial_i w\in H^1(\mesh)$, for all
$\inod$.

\begin{lemma}[Key identity]\label{Lemma:identity_II}
The following holds for all $v\in \calVq$ and all
$w\in H^2(\mesh) \cap H^1_0(\Omega)$:
\begin{equation} \label{eq:identity_II}
	\langle \Delta^2 v,w \rangle_{W^{-1,q},W_0^{1,q'}}
	= \su (\nabla^2 v, \nabla^2 {w}_K)_K
	- n_\sharp^{(4)}(v,w).
\end{equation}
\end{lemma}
\begin{proof}
Since $\Delta v\in W^{1,q}(\Omega)$, we have $\Delta^2v =\Delta(\Delta v)\in W^{-1,q}(\Omega)$. Moreover, since $w\in H^1_0(\Omega)$, we infer that
$$
\langle \Delta^2 v,w \rangle_{W^{-1,q},W_0^{1,q^\prime}}
=  -(\nabla \Delta v, \nabla {w})_{\Omega} =
\sum_{\inod} - (\Delta \partial_i v,\partial_i w)_\Omega.
$$
For all $\inod$, we have $\partial_iv\in\calVd$ and $\partial_iw\in H^1(\mesh)$.
Applying \eqref{eq:identity_2} to the right-hand side of the above equation, we obtain
$$
\langle \Delta^2 v,w \rangle_{W^{-1,q},W_0^{1,q^\prime}}= \sum_{\inod}
\su (\nabla\partial_i v,\nabla \partial_i w_K)_K -  \sum_{\inod} n_\sharp^{(2)}(\partial_i v,\partial_i w).
$$
The conclusion is straightforward by definition of the Hessian and of $n_\sharp^{(4)}$.
\end{proof}

The following reformulation of $n_\sharp^{(4)}$ will be useful in our analysis.
Let us set $n_{F,i}:=\n_F\SCAL\vece_i$ for all $\inod$, where
$(\vece_i)_{\inod}$ denotes the canonical Cartesian basis of $\mathbb{R}^d$.

\begin{lemma}[Reformulation of $n_\sharp^{(4)}$]
For all $v\in \calVq$ and all $w\in H^2(\mesh) \cap H^1_0(\Omega)$, we have
\begin{equation} \label{eq:def_n_sharp_1} \begin{aligned}
		n_\sharp^{(4)}(v, w)
		=  \su \sum_{\inod} \Big\{ & \sum_{F\in \FKi} \langle \gamma_{K,F}\upd( \nabla \partial_{i}v),  n_{F,i} \n_F\SCAL((\nabla w_K)_{|F}  -\mean{  \nabla w}_F) \rangle_F \\
		& + \sum_{F\in\FKb} \langle \gamma_{K,F}\upd( \nabla \partial_{i}v),  n_{F,i} \n_F\SCAL(\nabla w_K)_{|F}\rangle_F \Big\}.
\end{aligned}\end{equation}
\end{lemma}
\begin{proof}
We first use the definitions of $n_\sharp^{(2)}$ and $n_\sharp^{(4)}$ to write
\[
n_\sharp^{(4)}(v, w)
=  \su \sum_{\inod} \Big\{ \sum_{F\in \FKi} \langle \gamma_{K,F}\upd( \nabla \partial_{i}v),
(\partial_i w_K)_{|F}- \mean{\partial_iw}_F \rangle_F
+ \sum_{F\in\FKb} \langle \gamma_{K,F}\upd( \nabla \partial_{i}v),
(\partial_i w_K)_{|F} \rangle_F \Big\}.
\]
We observe that for all $K\in\mesh$ and all $F\in\FK$,
\[
(\partial_i w_K)_{|F} = \vece_i\SCAL(\nabla w_K)_{|F} = n_{F,i}\n_F\SCAL(\nabla w_K)_{|F} + \vece_i\SCAL(\nabla_t w_K)_{|F},
\]
where $(\nabla_t w_K)_{|F}$ denotes the tangential gradient of $w_K$ on $F$. Since $w\in H^1_0(\Omega)$, this quantity is single-valued on every mesh interface $F\in\Fint$ and vanishes on every mesh boundary face $F\in\Fb$. Moreover, since $\nabla \partial_i v$ has an integrable divergence on $\Omega$ by assumption, we infer that for all $F=\partial K_1 \cap \partial K_2 \in \Fint$,
\begin{equation}\label{eq:n_sharp_Fi}
	\sum_{j\in\{1,2\}} \langle \gamma_{K_j,F}\upd(\nabla\partial_i v), \phi \rangle_F = 0,
	\qquad \forall \phi \in W^{\frac{1}{\varrho},\varrho'}(F).
\end{equation}
This implies that
\[
\su \sum_{F\in \FK} \langle \gamma_{K,F}\upd( \nabla \partial_{i}v),(\partial_i w_K)_{|F} \rangle_F =
\su \sum_{F\in \FK} \langle \gamma_{K,F}\upd( \nabla \partial_{i}v),  n_{F,i} \n_F\SCAL(\nabla w_K)_{|F} \rangle_F.
\]
A similar reasoning shows that
\begin{equation} \label{eq:manip}
	\su \sum_{F\in \FKi} \langle \gamma_{K,F}\upd( \nabla \partial_{i}v), \mean{\partial_iw}_F \rangle_F =
	\su \sum_{F\in \FKi} \langle \gamma_{K,F}\upd( \nabla \partial_{i}v),  n_{F,i} \n_F\SCAL\mean{\nabla w} \rangle_F.
\end{equation}
Combining the two above identities proves the assertion.
\end{proof}

\begin{remark}[Simplifications]
Both terms in~\eqref{eq:manip} actually vanish owing to~\eqref{eq:n_sharp_Fi}. This means that for all $v\in\calVq$ and all $w\in H^2(\mesh)$, we have
\begin{subequations} \label{eq:n_sharp_simplif}
	\begin{equation}
		n_\sharp^{(4)}(v, w)
		=  \su \sum_{F\in \FK} \sum_{\inod} \langle \gamma_{K,F}\upd( \nabla \partial_{i}v), (\partial_i w_K)_{|F} \rangle_F,
	\end{equation}
	and whenever $w\in H^2(\mesh)\cap H^1_0(\Omega)$, we also have
	\begin{equation}
		n_\sharp^{(4)}(v, w)
		=  \su \sum_{F\in \FK} \sum_{\inod} \langle \gamma_{K,F}\upd( \nabla \partial_{i}v),  n_{F,i} \n_F\SCAL(\nabla w_K)_{|F}  \rangle_F.
\end{equation} \end{subequations}
However, the expressions~\eqref{eq:def_n_sharp} and~\eqref{eq:def_n_sharp_1} are those that are needed in this work because we will extend the domain of the first argument $v$ to spaces containing functions which are only piecewise smooth and for which \eqref{eq:n_sharp_Fi} no longer holds true. In this situation, it is important to use \eqref{eq:def_n_sharp} and~\eqref{eq:def_n_sharp_1} and not \eqref{eq:n_sharp_simplif}.
\end{remark}

\section{Discrete setting for $C^0$-HHO methods} \label{sec:discrete_setting}

In this section, we introduce the key ingredients to formulate $C^0$-HHO methods to approximate the biharmonic problem with both types of BC's. The starting point for the present $C^0$-HHO methods are the fully discontinuous HHO methods from \cite{DongErn2021biharmonic}. While the latter methods rely on a triple of discrete unknowns, approximating the solution in each mesh cell, its trace on each mesh face, and the trace of its normal derivative (oriented by $\n_F$) on each mesh face, the present $C^0$-HHO methods only rely on a pair of discrete unknowns, approximating the solution in each mesh cell and the trace of its normal derivative on each mesh face.

\subsection{Local reconstruction, stabilization, and stability}\label{sec:local}

Let $k\geq 0$ be the polynomial degree. Recall that we consider a mesh $\mesh$
from a shape-regular family of simplical meshes such that $\mesh$ covers $\Omega$ exactly.
For every mesh cell $K\in \mesh$, the local $C^0$-HHO space is
\begin{equation} \label{eq:def_fesE}
	\fesE: =\mathbb{P}_{k+2}(K) \times \mathbb{P}_{k}(\FK),
\end{equation}
with the broken polynomial space $\mathbb{P}_{k}(\FK) := \times_{F\in\FK} \mathbb{P}_k(F)$.
A generic element in  $\fesE$ is denoted $\hat{v}_K = (v_K,  \gamma_{\partial K})$ with $v_K \in \mathbb{P}_{k+2}(K)$ and  $\gamma_{\partial K} \in \mathbb{P}_{k}(\FK)$. The first component of $\hat{v}_K$ aims at representing the solution inside the
mesh cell and the second the trace of its normal derivative (oriented by $\n_K$) on the cell boundary.

Let $K\in\mesh$.
We define the local reconstruction operator $R_K: \fesE \rightarrow \mathbb{P}_{k+2}(K)$ such that, for all $\hat{v}_K \in \fesE$ with $\hat{v}_K:= (v_K, \gamma_{\partial K})$, the polynomial
$R_K (\hat{v}_K)\in \mathbb{P}_{k+2}(K)$ is uniquely defined by solving the following problem with test functions $w \in \mathbb{P}_{k+2}(K)^\perp:=\{w\in \mathbb{P}_{k+2}(K)\;|\; (w,\xi)_K=0, \forall \xi \in \mathbb{P}_{1}(K) \}$:
\begin{align}
	(\nabla^2 R_K(\hat{v}_K), \nabla^2 w)_{K}
	={}&	(\nabla^2  {v}_K, \nabla^2 w)_{K}
	- (\partial_n v_K - \gamma_{\partial K},  \partial_{nn}  w)_{\dK}, \label{eq: reconstruction}
\end{align}
together with the condition $(R_K (\hat{v}_K),\xi)_K = (v_K,\xi)_K$ for all $\xi\in \mathbb{P}_{1}(K)$.
The local stabilization bilinear form is defined such that, for all $(\hat{v}_K, \hat{w}_K)\in \fesE \times \fesE$, with  $\hat{v}_K:= (v_K, \gamma_{\partial K})$ and $\hat{w}_K:= (w_K, \chi_{\partial K})$,
\begin{equation}\label{def: stabilisation}
	\begin{aligned}
		S_{\dK}(\hat{v}_K,\hat{w}_K)
		:={}&   h_K^{-1}
		\big( \Pi^{k}_{\dK}(\gamma_{\partial K}- \partial_n {v}_K),
		\chi_{\partial K}- \partial_n {w}_K\big)_{\dK},
	\end{aligned}
\end{equation}
where $\Pi^k_{\dK}$ denotes the $L^2$-orthogonal projection onto $\mathbb{P}_k(\FK)$.

We define the local bilinear form ${a}_K$ on $\fesE \times \fesE$ such that
\begin{equation}\label{local bilinear form}
	{a}_K(\hat{v}_K,\hat{w}_K)
	:= (\nabla^2 R_K (\hat{v}_K),\nabla^2 R_K (\hat{w}_K))_{K}
	+S_{\dK}(\hat{v}_K,\hat{w}_K).
\end{equation}
We define the local energy seminorm such that, for all $\hat{v}_K:=(v_K,\gamma_{\partial K}) \in \fesE$,
\begin{align}
	|\hat{v}_K|^2_{\fesE}: ={} & \|\nabla^2 v_K\|_{K}^2
	+ h_K^{-1} \|\gamma_{\partial K} - \partial_n v_K\|_{\dK}^2. \label{H2_seminorm_elem}
\end{align}

\begin{lemma}[Local stability and boundedness] \label{lem: stability and boundedness}
	There is a real number $\alpha>0$, depending only on the mesh shape-regularity and the polynomial degree $k$, such that, for all $h>0$, all $K \in \mesh$, and all $\hat{v}_K\in \fesE$,
	\begin{equation}\label{local equivalent}
		\alpha|\hat{v}_K|^2_{\fesE}
		\leq  \|\nabla^2 R_K(\hat{v}_K) \|_{K}^2
		+ S_{\dK}(\hat{v}_K,\hat{v}_K)
		\leq 	\alpha^{-1} |\hat{v}_K|^2_{\fesE}.
	\end{equation}
\end{lemma}

\begin{proof}
	The proof proceeds as that of \cite[Lemma~4.1]{DongErn2021biharmonic}.
\end{proof}

\begin{remark}[Cell unknowns]
The choice of the polynomial space $\mathbb{P}_{k+2}(K)$ leverages on \cite{DongErn2021biharmonic} and has been made to allow for the simple stabilization bilinear form defined in~\eqref{def: stabilisation}. Alternative choices for the cell unknowns are possible, such as taking the polynomial space $\mathbb{P}_{k+1}(K)$ (see \cite[Remark~3.2]{DongErn2021biharmonic}) and even $\mathbb{P}_{k}(K)$ for $k\ge1$ (see \cite{BoDPGK:18}). Both choices lead to fewer cell unknowns, which can marginally alleviate the costs of static condensation, but require a more subtle form for stabilization involving the reconstruction operator, and this typically adds some computational costs. Moreover, whatever the choice for the cell unknowns, the size of the linear system after static condensation is the same.
\end{remark}

\subsection{Global discrete spaces and discrete problems}

Let $\mathbb{P}_{k+2}^{\rm g}(\mesh)$ denote the usual $C^0$-conforming finite element space  composed of continuous, piecewise polynomials of degree at most $(k+2)$ on the mesh $\mesh$ (the superscript ${}^{\rm g}$ refers to the integrability of the gradient of functions in $\mathbb{P}_{k+2}^{\rm g}(\mesh)$).
We set
\begin{equation}
	\fes := \mathbb{P}_{k+2}^{\rm g}(\mesh) \times V_{\Fall}^k, \qquad
	V_{\Fall}^k:=\mathbb{P}_k(\Fall),
\end{equation}
with the broken polynomial space $\mathbb{P}_k(\Fall):=\times_{F\in\Fall}\mathbb{P}_k(F)$.
We also define the following subspaces accounting for (homogeneous) boundary conditions:
\begin{equation}
	V_{\mesh,0}^{k+2} := \mathbb{P}_{k+2}^{\rm g}(\mesh)\cap H^1_0(\Omega),
	\qquad
	V_{\Fall,0}^k:=\{\gamma_{\Fall}\in V_{\Fall}^k\;|\; \gamma_{\Fall|F}=0,\, \forall F\in\Fb\}.
\end{equation}
The global $C^0$-HHO spaces used to approximate the biharmonic problems are then defined as follows:
\begin{equation} \label{eq:def_HHO_spaces}
	\fesI:= V_{\mesh,0}^{k+2} \times V_{\Fall,0}^k, \qquad
	\fesII:=V_{\mesh,0}^{k+2}\times V_{\Fall}^k,
\end{equation}
for type (I) and type (II) BC's, respectively. This means that the boundary condition $u=0$ is directly enforced on the trace of the cell unknowns at the boundary, the boundary condition $\partial_nu=0$ for type (I) BC's is directly enforced on the discrete unknowns representing the normal derivative at the mesh boundary faces, and the boundary condition $\partial_{nn}u=0$ is not enforced directly, but in a natural way (i.e., it results from the discrete problem). Notice also that $\fesI\subset \fesII$.

Any member of $\fesI$ or $\fesII$ is generically denoted by $\hat{v}_h:=(v_{\mesh},\gamma_{\Fall})$ with $v_{\mesh}:=(v_K)_{K\in\mesh}$ and $\gamma_{\Fall}:=(\gamma_F)_{F\in\Fall}$. For every mesh cell $K\in\mesh$, the local components of $\hat{v}_h$ are collected in the pair $\hat{v}_K:=(v_K,\gamma_{\dK})\in\fesE$ with $\gamma_{\dK}|_F: = (\n_F {\cdot}\n_K)\gamma_{F}$ for all $F \in \FK$.
The global bilinear form is assembled cellwise from the contributions of all the mesh cells, yielding
\begin{equation}
	{a}_h(\hat{v}_h, \hat{w}_h):= \su a_K(\hat{v}_K, \hat{w}_K).
\end{equation}
The discrete problems are as follows:
\begin{itemize}
	\item For type (I) BC's, one seeks $\hat{u}_h\upI\in \fesI$ such that
	\begin{equation}\label{discrete_pb_I}
		a_h(\hat{u}_h\upI,\hat{w}_h) = \ell(w_{\mesh}), \qquad \forall w_h\in \fesI.
	\end{equation}
	\item For type (II) BC's, one seeks $\hat{u}_h\upII\in \fesII$ such that
	\begin{equation}\label{discrete_pb_II}
		a_h(\hat{u}_h\upII,\hat{w}_h) = \ell(w_{\mesh}), \qquad \forall w_h\in \fesII.
	\end{equation}
\end{itemize}
Notice that the discrete problems~\eqref{discrete_pb_I} and \eqref{discrete_pb_II} employ the same discrete bilinear form $a_h$ and the same right-hand side; only the discrete trial and test spaces differ. We also observe that the above right-hand sides are meaningful since $V_{\mesh,0}^{k+2}\subset W^{1,q'}_0(\Omega)$.

Both spaces $\fesI$ and $\fesII$ are equipped with the norm
\begin{equation} \label{H2_seminorm_glob}
	\|\hat{v}_h\|_{\fes}^{2}:= \su |\hat{v}_K|^2_{\fesE},
\end{equation}
with the local energy seminorm $|\SCAL|_{\fesE}$ defined in~\eqref{H2_seminorm_elem}.
To verify that this indeed defines a norm on $\fesII$ (and thus also on $\fesI$), we notice that if $\hat{v}_h=(v_{\mesh},\gamma_{\Fall})\in\fesII$ satisfies $\|\hat{v}_h\|_{\fes}=0$, then $v_{\mesh}$ is a globally affine function in $\Omega$ which vanishes at the boundary $\partial\Omega$, so that $v_{\mesh}=0$; moreover, $\gamma_{\Fall}$ coincides on each mesh face with the trace of the normal derivative of $v_{\mesh}$, so that $v_{\Fall}=0$ as well. A direct consequence of the lower bound in~\eqref{local equivalent} is therefore that the global bilinear form $a_h$ is coercive on $\fesII$ (and thus also on $\fesI$). Hence, both discrete problems \eqref{discrete_pb_I} and~\eqref{discrete_pb_II} are well-posed owing to the Lax--Milgram lemma.

\begin{remark}[Static condensation]
	Unlike the fully discontinuous HHO methods from \cite{BoDPGK:18,DongErn2021biharmonic,DongErn2021singular} where all the the cell unknowns can be eliminated locally by a static condensation procedure, $C^0$-HHO methods are amenable to static condensation only if $(k+2)\geq (d+1)$ on $d$-dimensional simplices, and in this case, only the so-called cell bubble functions vanishing on the cell boundary can be eliminated, as in classical $C^0$-conforming FEM. Notice also that static condensation is more delicate for $C^0$-IPDG methods since only those bubble functions having a zero second-order derivative at the cell boundary can be locally eliminated.
\end{remark}

\subsection{Analysis tools and local HHO interpolation operator}\label{Inverse inequality and polynomial approximation}

For the reader's convenience, let us briefly restate some classical discrete inverse inequalities and polynomial approximation properties on shape-regular families of simplicial meshes. The results are classical, and we refer the reader, e.g., to \cite[Chap.~9-11]{Ern_Guermond_FEs_I_2021} for the proofs. In the rest of this paper, we use
the symbol $C$ to denote any positive generic constant
(its value can change at each occurrence)
that is independent of $h>0$, the considered mesh cell $K\in\mesh$, and the considered function in the inequality. The value of $C$ can depend on the shape-regularity parameter of the mesh sequence and the polynomial degree (whenever relevant).

\begin{lemma}[Discrete inverse inequalities]\label{lemma: Inverse inequality}
	Let $l\ge0$ be a polynomial degree.
	There is $C$ such that for all $h>0$, all $K\in\mesh$, all $p\in [1,\infty]$, and all $v_h\in\mathbb{P}_l(K)$,
	\begin{align}\label{trace inverse inequality}
		\|{v}_h\|_{\partial K} &\leq C h_K^{-\frac12}  \|{v}_h\|_{ K}, \\
		\label{Cell inverse inequality}
		\|\nabla {v}_h\|_{K} &\leq C h_K^{-1}  \|{v}_h\|_{ K},  \\
		\label{inverse inequality Lp}
		\| {v}_h\|_{L^{p}(K)} &\leq C h_K^{d(\frac1p-\frac12)}  \|{v}_h\|_{ K}.
	\end{align}
\end{lemma}

\begin{lemma}[Multiplicative trace inequality]\label{lemma: trace inequality}
	There is $C$ such that for all $h>0$,
	all $K\in\mesh$, and all $v\in H^1(K)$,
	\begin{equation}\label{trace inequality}
		\|{v} \|_{\dK} \leq C \big(
		h_K^{-\frac12}  \|{v}\|_{K} + h_K^{\frac12}  \|\nabla v\|_{K}\big).
	\end{equation}
\end{lemma}


\begin{lemma}[Polynomial approximation]\label{lemma:pol_app}
Let $l\geq 1$ be the polynomial degree and let $d\in\{2,3\}$ (so that $l+1>\frac{d}{2}$).
Let $\interL_h^{l}$ be the $H^1$-conforming Lagrange interpolation operator
onto $\mathbb{P}_{l}^{\rm g}(\mesh)$ and let $\interL_K^l$ be its local version mapping onto
$\mathbb{P}_{k+2}(K)$ for all $K\in\mesh$.
There is $C$ such that for all $r > \frac{d}{2}$,
all $m\in \{0,\ldots,\lfloor r\rfloor\}$, all $p\in [1,\infty]$, all $h>0$, all $K\in\mesh$,
and all $v\in H^r(K)$,
\begin{equation}\label{eq:pol_app}
h^m_K\|{v} - \interL_K^{l} (v)\|_{H^m(K)} + h^{d(\frac{1}{2}-\frac{1}{p})}_K\|{v} - \interL_K^{l} (v)\|_{L^p(K)}
\leq C h^{t}_K |{v} |_{H^t(K)},
\end{equation}
where $t:= \min\{r, l+1\}$.
\end{lemma}

Let us briefly highlight some useful consequences of Lemma~\ref{lemma:pol_app}. First,
taking $l=1$, $r=2$, and $m\in\{0,1\}$ in \eqref{eq:pol_app} shows that for all $h>0$,
all $K\in\mesh$, and all $v\in H^2(K)$,
\begin{equation}
	\|v-\interL_K^1(v)\|_K+h_K\|\nabla(v-\interL_K^1(v))\|_K\le Ch_K^2\|\nabla^2v\|_K.\label{eq:PK}
\end{equation}
Then, for all $v\in H^2(K)$, setting $z:=v-\interL_K^{k+2}(v)$ for any polynomial degree $k\ge0$, and observing that $\interL_K^1(z)=0$ since $\interL_K^1\circ \interL_K^{k+2}=\interL_K^1$, we infer from~\eqref{eq:PK} that $\|\nabla z\|_{K}\le Ch_K\|\nabla^2 z\|_K$. Hence, owing to the multiplicative trace inequality \eqref{trace inequality}, we have $h_K^{-\frac12}\|\nabla z\|_{\dK} \le C\|\nabla^2z\|_K$, i.e.,
\begin{align}
	h_K^{-\frac12}\|\nabla(v-\interL_K^{k+2}(v))\|_{\dK}
	\le C\|\nabla^2(v-\interL_K^{k+2}(v))\|_K. \label{eq:tr_PK}
\end{align}

\begin{remark}[$H^1$-conforming quasi-interpolation operator]
	The lower bound $r>\frac{d}{2}$ is not really a restriction for the biharmonic problem in space dimension $d\in\{2,3\}$ since the weak solution always sits in $H^2(\Omega)$. In higher space dimension, the restriction $r>\frac{d}{2}$ can be lifted by invoking a quasi-interpolation operator onto $\mathbb{P}_{l}^{\rm g}(\mesh)$ instead of the Lagrange interpolation operator; we refer the reader to \cite{ScoZh:90,ErnGuermond:17} for examples of quasi-interpolation operators that can be considered.
\end{remark}

For every mesh cell $K\in \mesh$, we define the local HHO reduction operator $\mathcal{\hat{I}}^k_K : H^2(K) \rightarrow \fesE$ such that, for all $v\in H^2(K)$,
\begin{equation}\label{def: reduction operator}
	\mathcal{\hat{I}}^k_K(v):= (\interL_{K}^{k+2}(v), \Pi_{\dK}^k (\n_K{\cdot}\nabla v)) \in \fesE.
\end{equation}
This definition is meaningful since $H^2(K)\hookrightarrow C^0(\overline{K})$ for $d\in\{2,3\}$.
In addition, we define the local HHO interpolation operator
\begin{equation}
	J\uHHO_K:= R_K \circ \mathcal{\hat{I}}^k_K: H^2(K) \rightarrow \mathbb{P}_{k+2} (K).
\end{equation}
In what follows, we need to measure the HHO interpolation error in some augmented norm which reflects Assumption~\ref{ass:regularity}. Specifically, we set
\begin{equation}
	\|v\|^2_{\sharp, K}:= \|\nabla^2 v\|_{K}^2 + h_K^{2d(\frac{1}{2}-\frac1p)}\|\nabla^2 v \|_{L^p(K)}^2  + {h_K^{2+2d(\frac{1}{2}-\frac1q)}\|\nabla \Delta v \|_{L^q(K)}^2}, \label{eq:def_sharp_II}
\end{equation}
for all $v\in W^{2,p}(K)$ with $\Delta v\in W^{1,q}(K)$ (recall that $p>2$ and $q\in (\frac{2d}{d+2},2]$).

\begin{lemma}[HHO interpolation error] \label{lem: approxmation}
In the above setting, the following holds for all $K\in \mesh$:
\begin{equation}\label{Error bound in sharp norm for Clamped plate}
\|v - J\uHHO_K (v)\|_{\sharp, K}^2 + S_{\dK}(\mathcal{\hat{I}}^k_K(v),\mathcal{\hat{I}}^k_K(v))
\leq
C \| v - \interL_K^{k+2}(v)\|_{\sharp,K}^2.
\end{equation}
\end{lemma}

\begin{proof}
(1) Let us first bound $\|v - J\uHHO_K (v)\|_{\sharp, K}$. The triangle inequality
followed by the discrete inverse inequalities from Lemma~\ref{lemma: Inverse inequality} implies that
\begin{align*}
\|v - J\uHHO_K (v)\|_{\sharp, K} &\leq \|v - \interL_K^{k+2}(v)\|_{\sharp, K} + \|J\uHHO_K(v) - \interL_K^{k+2}(v)\|_{\sharp, K} \\
&\leq \|v - \interL_K^{k+2}(v)\|_{\sharp, K} + C\|\nabla^2 (J\uHHO_K(v) - \interL_K^{k+2}(v))\|_{K},
\end{align*}
so that we only need to bound the last term on the right-hand side.
A straightforward calculation using the definition \eqref{eq: reconstruction} of $R_K$ shows that,
for all $w \in  \mathbb{P}_{k+2}(K)^\perp$,
\begin{align*}
(\nabla^2  J\uHHO_K(v), \nabla^2 w)_{K}
={}&	(\nabla^2  \interL_{K}^{k+2}(v), \nabla^2 w)_{K}
- (\partial_n \interL_{K}^{k+2}(v) - \Pi_{\dK}^k (\partial_n v),  \partial_{nn}  w)_{\dK} \\
={}&(\nabla^2  \interL_{K}^{k+2}(v), \nabla^2 w)_{K}
- (\partial_n (\interL_{K}^{k+2}(v) -   v) ,  \partial_{nn} w)_{\dK},
\end{align*}
together with the relation $(J\uHHO_K(v) ,  \xi)_{K} = (  \interL_{K}^{k+2}(v) ,  \xi)_{K} $, for all $\xi \in \mathbb{P}_1(K)$.
Re-arranging the terms, taking $w =  J\uHHO_K(v) - \interL_{K}^{k+2}(v)$, and
using the discrete trace inverse inequality \eqref{trace inverse inequality} gives
\begin{align*}
\|\nabla^2 (J\uHHO_K(v)  - \interL_{K}^{k+2}(v))\|_{K} \leq {}&
C h_K^{-\frac12}\|\partial_n (\interL_{K}^{k+2}(v) -   v)\|_{\dK}.
\end{align*}
Bounding the right-hand side using \eqref{eq:tr_PK} leads to
\[
\|\nabla^2 (J\uHHO_K(v)  - \interL_{K}^{k+2}(v))\|_{K} \leq {}
C \|\nabla^2 (v - \interL_{K}^{k+2}(v))\|_{K}.
\]
Putting the above bounds together shows that $\|v - J\uHHO_K (v)\|_{\sharp, K} \leq C
\|v - \interL_K^{k+2}(v)\|_{\sharp, K}$ since the $\|\SCAL\|_{\sharp,K}$-norm controls the $H^2(K)$-seminorm.

(2) Let us now bound $S_{\dK}(\mathcal{\hat{I}}^k_K(v),\mathcal{\hat{I}}^k_K(v))$.
We have
\begin{align*}
S_{\dK}(\mathcal{\hat{I}}^k_K(v),\mathcal{\hat{I}}^k_K(v))
={}& h_K^{-1} \| \Pi^k_{\dK}(\Pi^k_{\dK} (\partial_{n}v)- \partial_{n} (\interL^{k+2}_K (v))) \|^2_{\dK}
\leq h_K^{-1} \| \partial_{n}(v- \interL^{k+2}_K (v)) \|^2_{\dK},
\end{align*}
since $\Pi^k_{\dK}\circ \Pi^k_{\dK}=\Pi^k_{\dK}$ and $\Pi^{k}_{\partial K}$ is $L^2$-stable.
The  right-hand side is bounded by means of  \eqref{eq:tr_PK}, yielding
\begin{align*}
S_{\dK}(\mathcal{\hat{I}}^k_K(v),\mathcal{\hat{I}}^k_K(v)) &\leq
C \|\nabla^2 (v-\interL_{K}^{k+2}(v))\|_{K}^2.
\end{align*}
Since the $\|\SCAL\|_{\sharp,K}$-norm controls the $H^2(K)$-seminorm, this concludes the proof.
\end{proof}

It is convenient to define global versions of the above operators and norms.
The global HHO reduction operator $\mathcal{\hat{I}}_h^k:H^2(\Omega)\to \fes$ is defined such that, for all $v\in H^2(\Omega)$,
\begin{equation} \label{def:Ihk_bis}
	\Ihk(v):=\big( \interL_h^{k+2}(v),(\Pi_F^{k}(\n_F{\cdot}(\nabla v)_{|F}))_{F\in\Fall} \big)\in \fes,
\end{equation}
so that the local components of $\Ihk(v)$ are $\mathcal{\hat{I}}^k_K(v|_K)$ for all $K\in\mesh$. We also notice that $\Ihk(v)\in \fesII$ if $v\in H^1_0(\Omega)$ since the Lagrange interpolation operator preserves the homogeneous Dirichlet boundary condition. Moreover, we have $\Ihk(v)\in \fesI$ if $v\in H^2_0(\Omega)$. Furthermore, the global HHO interpolation operator mapping onto $\mathbb{P}_{k+2}(\mesh)$ is defined such that $J\uHHO_h(v)_{|K}:=J\uHHO_K(v_{|K})$ for all $K\in\mesh$ and all $v\in H^2(\Omega)$.
Finally, we define the global augmented norm such that
$\|v\|_{\sharp,h}^2:=\su \|v_{|K}\|_{\sharp,K}^2$ for all $v\in \calVq$.

\section{Error analysis} \label{sec:error_analysis}

In this section, we perform the error analysis under Assumption \ref{ass:regularity}. For completeness, we also outline the main arguments for  the error analysis of $C^0$-IPDG methods. In all cases, the main step is to bound the consistency error by making use of the key identity established in Lemma~\ref{Lemma:identity_II}.

\subsection{Preliminaries}

Before bounding the consistency error, we need to slightly adapt the bilinear form
$n_\sharp^{(4)}$ introduced in Section~\ref{sec:identity_II}. Recall the functional
space $\calVq \eqq \{v\in W^{2,p}(\Omega)\st \Delta v \in  W^{1,q}(\Omega)\}$, $p>2$, $q\in (\frac{2d}{d+2},2]$
(see \eqref{eq:calVq}) and that $n_\sharp^{(4)}$ is defined on $\calVq \times H^2(\mesh)$
using \eqref{eq:def_n_sharp} or equivalently \eqref{eq:def_n_sharp_1}.

The first adaptation is to allow for discrete functions as the
first argument of $n_\sharp^{(4)}$. Since functions in $\mathbb{P}_{k+2}^{\rm g}(\mesh)$ are
piecewise smooth, $n_\sharp^{(4)}$ can
be extended to $(\calVq+\mathbb{P}_{k+2}^{\rm g}(\mesh))\times H^2(\mesh)$. Moreover,
since $\gamma_{K,F}\upd(\nabla \partial_i v_h) = (\partial_n\partial_i v_{K})_{|F}$
for all $v_h\in \mathbb{P}_{k+2}^{\rm g}(\mesh)$ (with the notation $v_K:=v_{h|K}$),
a straightforward calculation
starting from \eqref{eq:def_n_sharp_1} shows that
for all $v_h,w_h\in \mathbb{P}_{k+2}^{\rm g}(\mesh)$,
\begin{equation}
n_\sharp^{(4)}(v_h,w_h) = \sum_{F\in\Fall} ( \mean{\partial_{nn}v_h}_F,\jump{\nabla w_h}_F\SCAL\n_F)_F,
\end{equation}
where the second-normal derivative is understood to act cellwise on $v_h$.
This identity is important when analyzing the $C^0$-IPDG method
(see Section~\ref{sec:C0IPDG}).

However, when analyzing the $C^0$-HHO method, a second adaptation is necessary since discrete test functions in the $C^0$-HHO method have two components, one attached to the mesh cells and one to the mesh faces, and not just one as in the $C^0$-IPDG method. Thus, inspired from~\eqref{eq:def_n_sharp_1}, we now introduce the following bilinear form on $(\calVq+\mathbb{P}_{k+2}^{\rm g}(\mesh))\times \fesII$ (recall that $\fesI \subset \fesII$):
\begin{equation}\label{eq:hat_n_sharp}
\hat{n}_\sharp^{(4)}(v,\hat{w}_h):= \su \sum_{F\in \FK} \sum_{\inod} \langle \gamma_{K,F}\upd(\nabla \partial_i v), n_{F,i} (\n_F\SCAL(\nabla w_K)_{|F} - \chi_F) \rangle_F.
\end{equation}
Recall the notation $\hat{w}_h=(w_{\mesh},\chi_{\Fall})$ with $w_{\mesh}=(w_K)_{K\in\mesh}$
and $\chi_{\Fall}=(\chi_F)_{F\in\Fall}$ and that the local components of $\hat{w}_h$
associated with the mesh cell $K\in\mesh$ are
$\hat{w}_K=(w_K,\chi_{\dK}=((\n_K\SCAL\n_F)\chi_F)_{F\in\FK})$.

\begin{lemma}[Identities for $\hat{n}_\sharp^{(4)}$]
\label{Lemma:identity_hat_n_sharp}
The following holds for all $v_h\in \mathbb{P}_{k+2}^{\rm g}(\mesh)$, all $v\in \calVq$, and all $\hat{w}_h \in \fesII$:
\begin{align}
\hat{n}_\sharp^{(4)}(v_h,\hat{w}_h)&
= \su(\nabla^2v_K,\nabla^2(w_K - R_K(\hat{w}_K)) )_K, \label{n sharp relation uh}\\
\hat{n}_\sharp^{(4)}(v,\hat{w}_h)&= n_\sharp^{(4)}(v,w_{\mesh}) -  \sum_{F\in \Fb} \sum_{\inod} \langle \gamma_{K,F}\upd(\nabla  \partial_i v),n_{F,i}  \chi_F \rangle_F \label{n sharp relation v in HHO},
\end{align}
where in the summation over all $F\in\Fb$, $K\in\mesh$ denotes the unique mesh cell such that $F=\partial K\cap \partial\Omega$.
\end{lemma}

\begin{proof}
(1) Proof of \eqref{n sharp relation uh}. Since $v_h$ is piecewise smooth, we have
$\gamma_{K,F}\upd(\nabla \partial_i v_K)_{|F}=(\n_K\SCAL\nabla\partial_i v_K)_{|F}$ for all $K\in\mesh$ and all $F\in\FK$, so that
\begin{equation*}
\sum_{\inod} \langle \gamma_{K,F}\upd(\nabla \partial_i v_K)_{|F}, n_{F,i}(\n_F\SCAL(\nabla w_K)_{|F} - \chi_F) \rangle_F
=\int_F  \partial_{nn}v_K(\partial_n w_K - \chi_{\dK}) \ud s,
\end{equation*}
where all the normal derivatives are understood to be oriented by $\n_K$.
Therefore, we obtain
\begin{equation*}
\begin{aligned}
\hat{n}_\sharp^{(4)}(v_h,\hat{w}_h) &
= \su  (\partial_{nn}{v_K}, \partial_n w_K - \chi_{\dK} )_{\partial K}.
\end{aligned}
\end{equation*}
Recalling the definition \eqref{eq: reconstruction} of the reconstruction operator $R_K$ applied to $\hat{w}_K$ proves \eqref{n sharp relation uh}.

(2) Proof of \eqref{n sharp relation v in HHO}. Using the expression~\eqref{eq:def_n_sharp_1} for $n_\sharp^{(4)}$, we infer that
\begin{align*}
\hat{n}_\sharp^{(4)}(v,\hat{w}_h)- n_\sharp^{(4)}(v,w_{\mesh}) =
\su \sum_{\inod} \Big\{ & \sum_{F\in\FKi} \langle \gamma_{K,F}\upd(\nabla\partial_iv),n_{F,i}(\n_F\SCAL\mean{\nabla w_{\mesh}}_F-\chi_F)\rangle_F \\
& - \sum_{F\in\FKb} \langle \gamma_{K,F}\upd(\nabla\partial_iv),n_{F,i}\chi_F\rangle_F\Big\}.
\end{align*}
Re-organizing the summation over the mesh interfaces yields
\begin{align*}
\hat{n}_\sharp^{(4)}(v,\hat{w}_h)- n_\sharp^{(4)}(v,w_{\mesh}) = {}& \sum_{F\in\Fint} \sum_{\inod} \sum_{j\in\{1,2\}}
\langle \gamma_{K_j,F}\upd(\nabla\partial_iv),n_{F,i}(\n_F\SCAL\mean{\nabla w_{\mesh}}_F-\chi_F)\rangle_F \\
& - \sum_{F\in\Fb} \sum_{\inod}
\langle \gamma_{K,F}\upd(\nabla\partial_iv),n_{F,i}\chi_F\rangle_F,
\end{align*}
recalling the notation $F=\partial K_1\cap \partial K_2$ for all $F\in\Fint$.
Since the first term on the right-hand side vanishes owing to \eqref{eq:n_sharp_Fi}, we conclude that \eqref{n sharp relation v in HHO} holds true.
\end{proof}

If the function $v$ is smooth, we have $\sum_{\inod} \langle \gamma_{K,F}\upd(\nabla \partial_i v), n_{F,i}  \chi_F  \rangle_F = \int_F (\partial_{nn}v) \chi_F \ud s.$ Thus, we expect that the right-hand side of \eqref{n sharp relation v in HHO} vanishes when applied with the first argument equal to the weak solution of \eqref{weak_form_II} (type (II) BC's). The same property holds when considering the weak solution of \eqref{weak_form_I} (type (I) BC's)
since, in this case, the face component of the test function vanishes at the boundary faces. Les us now formalize these arguments.

\begin{lemma}[Identity for weak solution]
\textup{(i)} Type \textup{(I)} BC's: Assume the weak solution $u\upI$ to \eqref{weak_form_I}
is in $\calVq$. Then, we have
\begin{subequations} \begin{equation} \label{eq:hat_n_sharp_u_I}
\hat{n}_\sharp^{(4)}(u\upI,\hat{w}_h)= n_\sharp^{(4)}(u\upI,w_{\mesh}), \qquad \forall \hat{w}_h \in \fesI.
\end{equation}
\textup{(ii)} Type \textup{(ii)} BC's: Assume the weak solution $u\upII$ to \eqref{weak_form_II}
is in $\calVq$. Then, we have
\begin{equation} \label{eq:hat_n_sharp_u_II}
\hat{n}_\sharp^{(4)}(u\upII,\hat{w}_h)= n_\sharp^{(4)}(u\upII,w_{\mesh}), \qquad \forall \hat{w}_h \in \fesII.
\end{equation} \end{subequations}
\end{lemma}

\begin{proof}
(1) Proof of \eqref{eq:hat_n_sharp_u_I}. The identity is a simple consequence of \eqref{n sharp relation v in HHO} and the fact that
$\chi_F=0$ for all $F\in\Fb$ whenever $\hat{w}_h=(w_{\mesh},\chi_{\Fall}) \in \fesI$.

(2) Proof of \eqref{eq:hat_n_sharp_u_II}. Let us prove that for all $F\in\Fb$, we have
\begin{equation} \label{eq:null_on_Fb}
\sum_{\inod} \langle \gamma_{K,F}\upd(\nabla \partial_i u\upII), n_{F,i}  \chi_F  \rangle_F = 0,
\end{equation}
where $K\in\mesh$ is the unique mesh cell such that $F=\partial K\cap \partial\Omega$.
Invoking \cite[p.~17]{GirRa:86}, we infer that there is $\phi_{\chi} \in H^1_0(K)\cap H^2(K)$, $(\partial_{n} \phi_{\chi})_{|F'} = 0$ for all $F'\in \FKi$ and $(\partial_{n} \phi_{\chi})_{|F} = \chi_F$. Let $\tilde \phi_{\chi}$ denote the zero-extension of $\phi_\chi$ to $H^1_0(\Omega)\cap H^2(\Omega)$. Using $\tilde\phi_\chi$ as test function in the weak formulation~\eqref{weak_form_II} and since $\Delta u\in W^{1,q}(\Omega)$ and $\tilde\phi_{\chi}\in W^{1,q'}_0(\Omega)$, we obtain
\begin{align*}
0 &=  (\nabla^2 u\upII, \nabla^2\tilde\phi_{\chi})_{\Omega} + (\nabla \Delta u\upII, \nabla \tilde\phi_{\chi})_{\Omega} \\
&=  (\nabla^2 u\upII, \nabla^2\phi_{\chi})_{K} + (\nabla \Delta u\upII, \nabla \phi_{\chi})_{K}
= \sum_{\inod} \Big\{(\nabla \partial_i u\upII, \nabla (\partial_i \phi_{\chi}))_{K} + (\Delta \partial_i u\upII, \partial_i \phi_{\chi})_{K}\Big\}.
\end{align*}
Moreover, we have
\begin{align}
\sum_{\inod} \langle \gamma_{K,F}\upd(\nabla \partial_i u\upII),n_{F,i}  \chi_F ) \rangle_F
&= \sum_{\inod} \langle \gamma_{K,F}\upd(\nabla \partial_i u\upII),(\partial_i \phi_{\chi})_{|F} \rangle_F  \\
&=\sum_{\inod} \Big\{(\nabla \partial_i u\upII, \nabla L_F^K(\partial_i \phi_{\chi} ))_{K} + (\Delta \partial_i u\upII, L_F^K(\partial_i \phi_{\chi} ))_{K}\Big\},\nonumber
\end{align}
where we used that the tangential derivative of $\phi_\chi$ vanishes on $F$ and the definition~\eqref{Def: trace operator} of the normal operator $\gamma_{K,F}\upd$.
Subtracting the above two relations gives
\begin{align}
\sum_{\inod} \langle \gamma_{K,F}\upd(\nabla \partial_i u\upII), n_{F,i}  \chi_F ) \rangle_F
&=\sum_{\inod} \Big\{(\nabla \partial_i u\upII, \nabla \delta_i)_{K} + (\Delta \partial_i u\upII, \delta_i )_{K}\Big\},
\end{align}
with $\delta_i: = L_F^K(\partial_i \phi_{\xi} ) - \partial_i \phi_{\xi}$.
We observe that $\delta_i \in W_0^{1,\varrho'}(K)\cap L^{2}(K)$. Considering a sequence
$(\delta_{i,\epsilon})_{\epsilon>0}$ in $C^\infty_0(K)$ which converges to $\delta_{i}$ in $W^{1,\varrho'}(K)\cap L^{2}(K)$ as $\epsilon\to0$ and since
$\sum_{\inod} \Big\{(\nabla (\partial_i u\upII), \nabla \delta_{i,\epsilon})_{K} + (\Delta (\partial_i u\upII), \delta_{i,\epsilon} )_{K}\Big\} = 0$ for all $\epsilon>0$, we conclude that \eqref{eq:null_on_Fb} holds true. Summing over the mesh boundary faces and invoking~\eqref{n sharp relation v in HHO} proves~\eqref{eq:hat_n_sharp_u_II}.
\end{proof}

We close this section by stating a boundedness estimate on the bilinear form $\hat{n}_\sharp^{(4)}$. We omit the proof since it follows the arguments from \cite[Lemma 3.2]{ErnGuer2021}.

\begin{lemma}[Boundedness of $\hat{n}_\sharp^{(4)}$]
\label{lem:bnd_n_sharp}
The following holds for all $v\in (\calVq+\mathbb{P}_{k+2}^{\rm g}(\mesh))$ and all $\hat{w}_h \in \fesII$:
\begin{equation}\label{error bound of n}
|\hat{n}_\sharp^{(4)}(v,\hat{w}_h)| \leq C \bigg(\su h_K^{2d(\frac{1}{2}-\frac1p)}\|\nabla^2 v \|_{L^p(K)}^2  + {h_K^{2+2d(\frac{1}{2}-\frac1q)}\|\nabla \Delta v \|_{L^q(K)}^2}\bigg)^{\frac12} \bigg(\su h_K^{-1} \|\chi_{\partial K} - \partial_n w_K\|_{\dK}^2 \bigg)^{\frac12}.
\end{equation}
\end{lemma}

\subsection{Bound on consistency error}

We define the consistency errors $\delta_h\upI\in (\fesI)^\prime$ and $\delta_h\upII\in (\fesII)^\prime$ such that
\begin{subequations}\begin{alignat}{2}
\langle \delta_h\upI,\hat{w}_h \rangle &:= {\ell}(w_{\mesh}) - a_h(\Ihk(u\upI), \hat{w}_h),
&\qquad &\forall \hat{w}_h\in \fesI,\\
\langle \delta_h\upII,\hat{w}_h \rangle &:= {\ell}(w_{\mesh}) - a_h(\Ihk(u\upII), \hat{w}_h),
&\qquad &\forall \hat{w}_h\in \fesII,
\end{alignat}\end{subequations}
where the brackets refer to the duality pairing between $(\fesI)^\prime$ and $\fesI$
or between $(\fesII)^\prime$ and $\fesII$ depending on the context.
Recall that the spaces $\fesI$ and $\fesII$ are equipped with the norm $\|\SCAL\|_{\fes}$ defined
in~\eqref{H2_seminorm_glob}. Recall also that the $\|\SCAL\|_{\sharp,K}$-norm is defined in~\eqref{eq:def_sharp_II}.

\begin{lemma}[Consistency]\label{lemma: consistency simply supported}
	Let $\delta_h$ denote either $\delta_h\upI$ or $\delta_h\upII$,
	let $u$ denote either $u\upI$ or $u\upII$, and let
	$\fes$ denote either $\fesI$ or $\fesII$. Assume that $u\in\calVq$.
	The following holds:
	\begin{equation} \label{consistency simply supported}
		\langle \delta_h,\hat{w}_h \rangle
		\leq C
		\left( \su \|u-\interL_K^{k+2} (u)\|^2_{\sharp,K}\right)^{\frac12} \|\hat{w}_h\|_{\fes},
		\qquad \forall \hat{w}_h\in \fes.
	\end{equation}
\end{lemma}
\begin{proof}
	Using the key identity \eqref{eq:identity_II} from Lemma~\ref{Lemma:identity_II} followed by the identity \eqref{eq:hat_n_sharp_u_I} or the identity \eqref{eq:hat_n_sharp_u_II} from Lemma~\ref{Lemma:identity_hat_n_sharp} depending on the context, we infer that for all $\hat{w}_h\in \fes$,
	\begin{equation*}
		\ell(w_{\mesh}) = \langle f,w_{\mesh} \rangle_{W^{-1,q},W_0^{1,q^\prime}} =\su  (\nabla^2 u, \nabla^2 {w}_K)_K- n_\sharp^{(4)}(u, w_{\mesh})
		= \su  (\nabla^2 u, \nabla^2 {w}_K)_K	- \hat{n}_\sharp^{(4)}(u,\hat{w}_h).
	\end{equation*}
	Moreover, using the identity \eqref{n sharp relation uh} from Lemma~\ref{Lemma:identity_hat_n_sharp} gives
	\begin{align*}
		a_h(\Ihk(u), \hat{w}_h) =
		\su \Big\{ (\nabla^2 J\uHHO_K(u), \nabla^2 {w}_K)_K
		+ S_{\dK}(\mathcal{\hat{I}}^k_K(u),\hat{w}_K) \Big\} - \hat{n}_\sharp^{(4)}(J\uHHO_h(u),\hat{w}_h).
	\end{align*}
	Thus, defining the function $\eta:=u-J\uHHO_h(u)$, i.e., $\eta|_K:=u|_K-J\uHHO_K(u_{|K})$ for all $K\in\mesh$, we infer that
	\begin{equation}\label{the error relation SP}
		\begin{aligned}
			\langle \delta_h,\hat{w}_h \rangle
			= {}& \su \Big\{ (\nabla^2 \eta, \nabla^2 {w}_K)_K
			- S_{\dK}(\mathcal{\hat{I}}^k_K(u),\hat{w}_K)\Big\}
			- \hat{n}_\sharp^{(4)}(\eta,\hat{w}_h).
		\end{aligned}
	\end{equation}
	Let us denote by $\mathcal{T}_{1}$ the first two terms on the right-hand side and by $\mathcal{T}_{2}$ the third addend. We bound $\mathcal{T}_{1,K}$ by the same arguments as above, yielding
	\begin{equation*}
		|\mathcal{T}_{1}|\le  C\left( \su \| \nabla^2 \eta \|^2_{K} + \| \nabla^2(u - \interL_K^{k+2}(u))\|_{K}^2
		\right)^{\frac12} \|\hat{w}_h\|_{\fes}.
	\end{equation*}
	Moreover, owing to~\eqref{error bound of n}, we have
	\begin{equation*}
		|\mathcal{T}_{2}| \le C \Big(\su h_K^{2d(\frac{1}{2}-\frac1p)}\|\nabla^2 \eta \|_{L^p(K)}^2  + {h_K^{2+2d(\frac{1}{2}-\frac1q)}\|\nabla \Delta \eta \|_{L^q(K)}^2}\Big)^{\frac12}\Big(\su h_K^{-1} \|\chi_{\partial K} - \partial_n w_K\|_{\dK}^2 \Big)^{\frac12}.
	\end{equation*}
	Altogether, this implies that
	\begin{equation*}
		|\langle \delta_h,\hat{w}_h \rangle| \le C\left( \su \| \eta \|^2_{\sharp,K} + \| \nabla^2(u - \interL_K^{k+2}(u))\|_{K}^2
		\right)^{\frac12} \|\hat{w}_h\|_{\fes}.
	\end{equation*}
	Invoking Lemma~\ref{lem: approxmation} completes the proof.
\end{proof}

\begin{remark}[Classical regularity assumption] \label{rem:classical_reg}
For completeness, let us briefly sketch how the consistency error is bounded under
the classical regularity assumption $u\in H^{2+s}(\Omega)$ and $f\in H^{-2+s}(\Omega)$, $s\in (\frac12,1]$. As above, we let $u$ denote either $u\upI$ or $u\upII$, and
$\fes$ denote either $\fesI$ or $\fesII$.
Starting from the identity from Lemma~\ref{Lemma:identity_I} and using that
$\partial_{nn}u$ is single-valued at every mesh interface $\FKi$ and that either $\partial_{nn}u$ or $\chi_{\dK}$ vanish at every mesh boundary face $F\in\FKb$ depending on the type of BC that is enforced, we infer that for all $\hat{w}_h\in \fes$,
\begin{equation*}
\langle f,w_{\mesh} \rangle_{H^{-2+s},H^{2-s}_0} =
\su \Big\{ (\nabla^2 u, \nabla^2 {w}_K)_K- (\partial_{nn}  u , \partial_{n} w_K  - \chi_{\dK})_{\dK} \Big\}.
\end{equation*}
Furthermore, using the identity $R_K\circ \mathcal{\hat{I}}^k_K = J\uHHO_K$ from
Lemma~\ref{lem: approxmation} together with the definition \eqref{eq: reconstruction} of
$R_K(\hat{w}_K)$ leads to
\begin{align*}
a_h(\Ihk(u), \hat{w}_h) ={}
\su \Big\{ (\nabla^2 J\uHHO_K(u), \nabla^2 {w}_K)_K
- (\partial_{nn} J\uHHO_K(u), \partial_{n} w_K  - \chi_{\dK})_{\dK}
+ S_{\dK}(\mathcal{\hat{I}}^k_K(u),\hat{w}_K) \Big\}.
\end{align*}
Defining the function $\eta$ cellwise as $\eta|_K:=u|_K-J\uHHO_K(u)$ for all $K\in\mesh$,
we infer that
\begin{align*}
\langle \delta_h,\hat{w}_h \rangle
:= {}& \langle f,w_{\mesh} \rangle_{H^{-2+s},H^{2-s}_0}  - a_h(\Ihk(u), \hat{w}_h) \\
= {}& \su \Big\{ (\nabla^2 \eta, \nabla^2 {w}_K)_K
- (\partial_{nn} \eta, \partial_{n} w_K  - \chi_{\dK})_{\dK}
- S_{\dK}(\mathcal{\hat{I}}^k_K(u),\hat{w}_K)\Big\}.
\end{align*}
Setting $\|v\|^2_{\sharp, K} := \|\nabla^2 v\|_{K}^2 +h_K \|\partial_{nn}  v \|_{\partial K}^2$, it is then straightforward to establish that
\begin{equation*}
\langle \delta_h,\hat{w}_h \rangle
\leq C \left( \su \|u-\interL_K^{k+2} (u)\|^2_{\sharp,K}\right)^{\frac12} \|\hat{w}_h\|_{\fes}.
\end{equation*}
\end{remark}

\subsection{Energy-error estimate}

We are now ready to establish our main convergence result bounding the error in the energy norm.
For completeness, improved error estimates in weaker norms are
outlined in Section~\ref{sec: weaker_norm_error}. Let us define the discrete error
$\hat{e}_h : = \Ihk(u) - \hat{u}_h \in \fes$ with local components $\hat{e}_K:=
\mathcal{\hat{I}}^k_K(u)-\hat{u}_K\in \fesE$ for all $K\in\mesh$.

\begin{theorem}[Error estimate]\label{Theorem:error_II}
Let $u$ denote either the weak solution $u\upI$ of \eqref{weak_form_I} or
the weak solution $u\upII$ of \eqref{weak_form_II}.
Let $\hat{u}_h$ denote either the discrete solution $\hat{u}_h\upI$ of \eqref{discrete_pb_I}
or the discrete solution $\hat{u}_h\upII$ of \eqref{discrete_pb_II}, respectively.
Under Assumption~\ref{ass:regularity}, the following holds:
\begin{equation}\label{eq:err1 Simply supported}
\su \Big\{ \|u- R_K(\hat{u}_K)\|_{\sharp,K}^2 + S_{\partial K}(\hat{e}_K,\hat{e}_K)
\Big\} \leq C \su \| u-\interL_K^{k+2}(u)\|_{\sharp,K}^2.
\end{equation}
Moreover, assuming $u\in H^{2+s}(\Omega)$ with $s>0$, and letting $t:= \min\{s,k+1\}$ as well as $\mathfrak{I}_t:=1$ if {$t < 1+d(\frac12-\frac1q)$} and $ \mathfrak{I}_t:=0$ otherwise, we have
\begin{equation}\label{eq:err2 Simply supported}
\su \Big\{  \|u- R_K(\hat{u}_K)\|_{\sharp,K}^2 +  S_{\partial K}(\hat{e}_K,\hat{e}_K) \Big\}
\leq C \su \Big( h_K^{t}|u|_{H^{t+2}(K)}
+ \mathfrak{I}_t {h_K^{1+d(\frac{1}{2}-\frac1q)}\|\nabla \Delta v \|_{L^q(K)}^2}\Big)^2.
\end{equation}
\end{theorem}

\begin{proof}
(1) Let $\delta_h$ denote either $\delta_h\upI$ or $\delta_h\upII$,
and let $\fes$ denote either $\fesI$ or $\fesII$ depending on the context.
Since $a_h(\hat{e}_h,\hat{e}_h)= -\langle \delta_h,\hat{e}_h \rangle$, the coercivity of the bilinear form $a_h$ implies that
\begin{equation*}
\alpha \|\hat{e}_h\|_{\fes}^2 \le a_h(\hat{e}_h,\hat{e}_h)= -\langle \delta_h,\hat{e}_h \rangle
\le \|\delta_h\|_{(\fes)'} \|\hat{e}_h\|_{\fes},
\end{equation*}
so that $\|\hat{e}_h\|_{\fes}\le \frac{1}{\alpha}\|\delta_h\|_{(\fes)'}$.
Since $\su \{ \|\nabla^2 R_K(\hat{e}_K)\|_{\sharp,K}^2 + S_{\partial K}(\hat{e}_K,\hat{e}_K)\} \le C\su \{\|\nabla^2 R_K(\hat{e}_K)\|_K^2+S_{\partial K}(\hat{e}_K,\hat{e}_K)\} \le C \|\hat{e}_h\|_{\fes}^2$ owing to the inverse inequality \eqref{inverse inequality Lp} and the upper bound in~\eqref{local equivalent}, we infer from Lemma \ref{lemma: consistency simply supported} that
\begin{equation*}
\su \Big\{ \|\nabla^2 R_K(\hat{e}_K)\|_{\sharp, K}^2 + S_{\partial K}(\hat{e}_K,\hat{e}_K) \Big\} \leq C \su \|u-\interL_K^{k+2} (u)\|^2_{\sharp,K}.
\end{equation*}
Since $u-R_K(\hat{u}_K)=(u-J\uHHO_K(u))+R_K(\hat{e}_K)$,
the triangle inequality combined with Lemma~\ref{lem: approxmation} and the above bound
proves~\eqref{eq:err1 Simply supported}.

(2) The estimate \eqref{eq:err2 Simply supported} results from \eqref{eq:err1 Simply supported} and the approximation properties of $\interL_K^{k+2}$ established in Lemma~\ref{lemma:pol_app}.
\end{proof}

\begin{remark}[Variant] \label{rem:variant_H2}
Invoking inverse inequalities shows that for all $K\in\mesh$,
\begin{equation} \label{eq:bound_uK_RuK}
\|u_K- R_K(\hat{u}_K)\|_{\sharp,K} \le C\|\nabla^2(u_K- R_K(\hat{u}_K))\|_{K}
\le CS_{\partial K}(\hat{u}_K,\hat{u}_K),
\end{equation}
where the second bound follows from $(\nabla^2(u_K- R_K(\hat{u}_K)),\nabla^2w)_K =
(\partial_nu_K-\gamma_{\dK},\partial_{nn}w)_{\dK}$ for all $w\in \mathbb{P}_{k+2}(K)^\perp$,
a discrete trace inequality, $\partial_{nn}w\in \mathbb{P}_k(\FK)$, and the definition of the stabilization operator.
Moreover, the triangle inequality
combined with the bound~\eqref{Error bound in sharp norm for Clamped plate} from Lemma~\ref{lem: approxmation} implies that
\[
\frac{1}{2}S_{\partial K}(\hat{u}_K,\hat{u}_K) \le S_{\dK}(\mathcal{\hat{I}}^k_K(u),\mathcal{\hat{I}}^k_K(u)) + S_{\partial K}(\hat{e}_K,\hat{e}_K) \le C\| u-\interL_K^{k+2}(u)\|_{\sharp,K} + S_{\partial K}(\hat{e}_K,\hat{e}_K).
\]
This leads to the following variant of the above energy-error estimate:
\begin{equation}
\su \Big\{ \|\nabla^2(u- u_K)\|_{K}^2 + S_{\partial K}(\hat{u}_K,\hat{u}_K)
\Big\} \leq C \su \| u-\interL_K^{k+2}(u)\|_{\sharp,K}^2,
\end{equation}
which follows from~\eqref{eq:err1 Simply supported}, the triangle inequality, and the above bounds.
\end{remark}

\subsection{Energy-error estimate for $C^0$-IPDG methods}
\label{sec:C0IPDG}

To illustrate that the key identity \eqref{eq:identity_II} from Lemma~\ref{Lemma:identity_II} has a broader applicability than $C^0$-HHO methods, we briefly outline here how this identity can be used in the error analysis of $C^0$-IPDG methods.
Focusing for brevity on type (II) BC's only, the $C^0$-IPDG bilinear form is defined as follows
for all $v_h, w_h \in V_{\mesh,0}^{k+2} := \mathbb{P}_{k+2}^{\rm g}(\mesh)\cap H^1_0(\Omega)$:
\begin{equation} \begin{aligned}
		a_h^{\textsc{ipdg}}(v_h, w_h):=& \su (\nabla^2 v_h,\nabla^2 w_h)_K + \sum_{F\in \Fint} \varpi_F h_F^{-1} (\jump{\partial_n v_h}_F,\jump{\partial_n w_h}_F)_F\\
		& - \sum_{F\in \Fint} (\mean{\partial_{nn} v_h}_F,\jump{\partial_n w_h}_F)_F - \sum_{F\in \Fint} (\jump{\partial_n v_h}_F,\mean{\partial_{nn}w_h}_F)_F,
\end{aligned}\end{equation}
where all the normal derivatives are understood to be oriented by $\n_F$ and
with the user-defined penalty parameter $\varpi_F>0$. For type (I) BC's, the
three summations over the mesh interfaces are realized over the whole set of mesh faces.
The discrete problem consists of
finding $u_h \in V_{\mesh,0}^{k+2}$ such that
\begin{equation}
	a_h^{\textsc{ipdg}}(v_h, w_h)=\ell(w_h):= \langle f,w_h \rangle_{W^{-1,q},W_0^{1,q^{\prime}}}, \qquad \forall w_h\in V_{\mesh,0}^{k+2}.
\end{equation}
The stability analysis reveals that the bilinear form $a_h^{\textsc{ipdg}}$ is coercive on
$V_{\mesh,0}^{k+2}$ if the penalty parameters $\varpi_F$ are large enough; see, e.g., \cite{BrennerC0}. The coercivity norm is $\|w_h\|_{V_{\mesh,0}^{k+2}}^2:= \su \|\nabla^2 w_h\|_K^2 + \sum_{F\in \Fint}h_F^{-1}\|\jump{\partial_n w_h}_F\|_F^2$.

The novelty here consists in bounding the consistency error under Assumption~\ref{ass:regularity}.
Let the bilinear form $n_\sharp^{\textsc{ipdg}}$ be defined as $n_\sharp^{(4)}$, except that the contributions of all the mesh boundary faces $F\in\Fb$ are discarded. The arguments in the proof of \eqref{eq:hat_n_sharp_u_II} show that $n_\sharp^{\textsc{ipdg}}(u,w_h) = n_\sharp^{(4)}(u,w_h)$ for all $w_h\in V_{\mesh,0}^{k+2}$. Therefore, using the key identity \eqref{eq:identity_II} from Lemma~\ref{Lemma:identity_II} gives
\begin{equation}
	\langle f,w_h \rangle_{W^{-1,q},W^{1,q'}_0} =  \su  (\nabla^2 u, \nabla^2 {w}_h)_K- n_\sharp^{(4)}(u, w_h)= \su  (\nabla^2 u, \nabla^2 {w}_h)_K- n_\sharp^{\textsc{ipdg}}(u, w_h),
\end{equation}
for all $w_h\in V_{\mesh,0}^{k+2}$. Moreover,
using the same proof as for~\eqref{n sharp relation uh}, we establish the following identity for all $v_h, w_h \in V_{\mesh,0}^{k+2}$:
\begin{equation}
	n_\sharp^{\textsc{ipdg}}(v_h, w_h) = \sum_{F\in \Fint} (\mean{\partial_{nn} v_h}_F,\jump{\partial_n w_h}_F)_F.
\end{equation}
Putting everything together and setting $\eta:= u-\interL_h^{k+2}(u)$, we can express the consistency error as follows:
\begin{align}
	\langle \delta_h,w_h \rangle
	:= {}& \ell(w_h) - a_h^{\textsc{ipdg}}(\interL_h^{k+2}(u),w_h)
	\nonumber \\
	= {}& \su (\nabla^2 \eta, \nabla^2 {w}_h)_K
	+ \sum_{F\in \Fint} \varpi_F h_F^{-1}
	(\jump{\partial_n \eta}_F,\jump{\partial_n w_h}_F)_F \nonumber \\
	&- n_\sharp^{\textsc{ipdg}}(\eta,w_h) - \sum_{F\in \Fint} (\jump{\partial_n \eta}_F,\mean{\partial_{nn} w_h}_F)_F,
\end{align}
where we used that $\jump{\partial_n u}_F=0$ for all $F\in\Fint$.
Using the same $\|\SCAL\|_{\sharp,K}$-norm as for the $C^0$-HHO method, and invoking the Cauchy--Schwarz inequality, Lemma~\ref{lem:bnd_n_sharp} to bound $n_\sharp^{\textsc{ipdg}}(\eta,w_h)$, and the estimate~\eqref{eq:tr_PK} to estimate $\partial_n\eta$ at the mesh interfaces, we infer that
\begin{equation}
	\langle \delta_h,w_h \rangle
	\leq C
	\left( \su \|u-\interL_K^{k+2} (u)\|^2_{\sharp,K}\right)^{\frac12} \|w_h\|_{V_{\mesh,0}^{k+2}},
	\qquad \forall w_h\in  V_{\mesh,0}^{k+2}.
\end{equation}
Finally, proceeding as above for the $C^0$-HHO method readily leads to the same error estimates.

\subsection{Improved error estimates in weaker norms} \label{sec: weaker_norm_error}

To derive error estimates in weaker norms, we assume that the following regularity pickup holds true: There exists a constant $C_{\Omega}$ such that for all $g\in H^\sigma(\Omega)$, $\sigma\in\{-1,0\}$, the adjoint solution such that $\Delta^2 \zeta_g = g$ in $\Omega$ with either type (I) or (II) BC's satisfies the bound
\begin{equation} \label{eq:pickup}
\|\zeta_g\|_{H^{4+\sigma}(\Omega)} \leq C_{\Omega} \|g\|_{H^\sigma(\Omega)}, \qquad \forall \sigma\in\{-1,0\}.
\end{equation}
A sufficient condition for \eqref{eq:pickup} to hold with type (I) BC's is that the domain $\Omega$ is convex; see \cite[p.~182]{Kozlov-Mazya:1996}.
In what follows, we assume $f\in L^2(\Omega)$ so that the above regularity pickup implies that $u\in H^4(\Omega)$. Hence, the setting of Assumption~\ref{ass:regularity} is not needed here, and we simply use the $\|\cdot\|_{\sharp,K}$-norm already considered in Remark~\ref{rem:classical_reg}, that is, $\|v\|^2_{\sharp, K} := \|\nabla^2 v\|_{K}^2 +h_K \|\partial_{nn}  v \|_{\partial K}^2$ for all $K\in\mesh$. We also set $\|v\|_{\sharp,h}^2:=\su \|v_{|K}\|_{\sharp,K}^2$. Let $h:=\max_{K\in \mesh} h_K$ denote the (global) mesh size and let $\tilde{h}$ be the piecewise constant function such that $\tilde{h}|_K:=h_K$ for all $K\in\mesh$.

\begin{theorem}[$L^2$-and $H^1$-error estimate]\label{Theorem:error_L2_H1}
Let $u$ denote either the weak solution $u\upI$ of \eqref{weak_form_I} or
the weak solution $u\upII$ of \eqref{weak_form_II}.
Let $\hat{u}_h$ denote either the discrete solution $\hat{u}_h\upI$ of \eqref{discrete_pb_I}
or the discrete solution $\hat{u}_h\upII$ of \eqref{discrete_pb_II}, respectively.
Assume $f\in  L^{2}(\Omega)$. Then, if the regularity estimate~\eqref{eq:pickup} holds true
with $\sigma=0$, we have
\begin{equation}\label{eq:L2err error}
\su \|u- R_K(\hat{u}_K)\|_{K}^2 \leq Ch^{2\min(k+1,2)} \Big( \|\hat{e}_h\|_{\fes}^2 + \| u-\interL_h^{k+2}(u)\|_{\sharp,h}^2 + \|\tilde{h}^2(f-\ \Pi^{k-2}_{h}(f))\|_{\Omega}^2 \Big),
\end{equation}
and if the regularity estimate~\eqref{eq:pickup} holds true
with $\sigma=-1$, we have
\begin{equation}\label{eq:H1err2 error}
\su \|\nabla (u- R_K(\hat{u}_K))\|_{K}^2 \leq Ch^2 \Big( \|\hat{e}_h\|_{\fes}^2 + \| u-\interL_h^{k+2}(u)\|_{\sharp,h}^2 + \|\tilde{h}^2(f-\ \Pi^{k-2}_{h}(f))\|_{\Omega}^2 \Big).
\end{equation}
\end{theorem}

To prove Theorem~\ref{Theorem:error_L2_H1}, we need to consider a novel HHO interpolation
operator. For all $K\in \mesh$, let $\interC_{K}^{k+2}:H^2(K)\rightarrow \mathbb{P}_{k+2}(K)$ be
the local interpolation operator associated with the canonical hybrid finite element in $K$; see \cite[section 7.6]{Ern_Guermond_FEs_I_2021} (this is the finite element considered in the first step of the discrete de Rham diagram). Recall that for $d=3$ and a function $v\in H^2(K)$, $\interC_{K}^{k+2}(v)$ is uniquely defined by the following properties:
\begin{subequations} \label{eq:def_canonical_FEM} \begin{alignat}{2}
\label{vertex property}
\interC_{K}^{k+2}(v) (\vtx) & =  v(\vtx), &\qquad& \forall \vtx \in \mathcal{V}_K, \\
\label{edge property}
(\interC_{K}^{k+2}(v), \xi_E)_E &=  (v, \xi_E)_E, &\qquad& \forall \xi_E \in\mathbb{P}_{k}(E), \forall E \in \mathcal{E}_K,\\
\label{face property}
(\interC_{K}^{k+2}(v), \xi_F)_F &=  (v, \xi_F)_F, &\qquad& \forall \xi_F \in\mathbb{P}_{k-1}(F), \forall F \in \mathcal{F}_K,\; k\ge1,\\
\label{cell property}
(\interC_{K}^{k+2}(v), \xi_K)_K &=  (v, \xi_K)_K, &\qquad& \forall   \xi_K \in\mathbb{P}_{k-2}(K), \; k\ge 2,
\end{alignat} \end{subequations}
where $\mathcal{V}_K$, $\mathcal{E}_K$, and $\mathcal{F}_K$ collect the vertices, edges, and faces of $K$, respectively. For $d=2$, $\interC_{K}^{k+2}(v)$ is defined similarly by using \eqref{vertex property}, \eqref{edge property}, and \eqref{cell property} with $\xi_K \in\mathbb{P}_{k-1}(K)$, $k\ge1$. The following holds for all $K\in\mesh$ and all $v\in H^3(K)$:
\begin{equation} \label{eq:inter_C_L}
\| v - \interC_K^{k+2}(v)\|_{\sharp,K} \le C \| v - \interL_K^{k+2}(v)\|_{\sharp,K}.
\end{equation}
This follows from the triangle inequality, inverse inequalities yielding $\| \interC_K^{k+2}(v) - \interL_K^{k+2}(v)\|_{\sharp,K}\le C\| \nabla^2(\interC_K^{k+2}(v) - \interL_K^{k+2}(v))\|_{K}$, and the $H^2$-stability of $\interC_K^{k+2}$ together with the fact that $\interC_K^{k+2}(v) - \interL_K^{k+2}(v) = \interC_K^{k+2}(v- \interL_K^{k+2}(v))$.

We then define the new local HHO reduction operator $\widehat{\mathcal{P}}^k_K : H^2(K) \rightarrow \fesE$ such that, for all $v\in H^2(K)$,
\begin{equation}\label{def: new reduction operator}
	\widehat{\mathcal{P}}^k_K(v):= (\interC_{K}^{k+2}(v), \Pi_{\dK}^k (\n_K{\cdot}\nabla v)) \in \fesE.
\end{equation}
as well as the new local HHO interpolation operator such that
\begin{equation}
T\uHHO_K:= R_K \circ \widehat{\mathcal{P}}^k_K: H^2(K) \rightarrow \mathbb{P}_{k+2} (K).
\end{equation}
The main motivation for the above construction is the following result.

\begin{lemma}[New HHO interpolation operator]
The following holds for all $K\in\mesh$ and all $v\in H^2(K)$:
\begin{equation} \label{eq:H2_ell_proj}
(\nabla^2 T\uHHO_K(v), \nabla^2 w)_{K} = (\nabla^2 v, \nabla^2 w)_{K},
\qquad \forall w \in \mathbb{P}_{k+2}(K)^\perp.
\end{equation}
Moreover, we have for all $K\in\mesh$ and all $v\in H^3(K)$,
\begin{equation}\label{Error bound in sharp norm for elliptic projection}
\|v - T\uHHO_K (v)\|_{\sharp, K}^2 + S_{\dK}(\widehat{\mathcal{P}}^k_K(v),\widehat{\mathcal{P}}^k_K(v))
\leq
C \| v - \interL_K^{k+2}(v)\|_{\sharp,K}^2.
\end{equation}
\end{lemma}

\begin{proof}
(1) Proof of \eqref{eq:H2_ell_proj}. Using the definition of the reconstruction operator $R_K$, the identity $\nabla \cdot \nabla^2 = \nabla \Delta$ and integration by parts gives
\begin{align*}
&(\nabla^2 T\uHHO_K (v), \nabla^2 w)_{K} \\
&=(\nabla^2  \interC_{K}^{k+2}(v), \nabla^2 w)_{K}
- (\partial_n \interC_{K}^{k+2}(v) - \Pi_{\dK}^k (\n_K{\cdot}\nabla v),  \partial_{nn}  w)_{\dK}\\
&=-(\nabla  \interC_{K}^{k+2}(v), \nabla \Delta w)_{K}
+ (\Pi_{\dK}^k (\n_K{\cdot}\nabla v),  \partial_{nn}  w)_{\dK}
+  (\partial_t \interC_{K}^{k+2}(v),  \partial_{nt}  w)_{\dK}\\
&=( \interC_{K}^{k+2}(v),  \Delta^2 w)_{K}
-(\interC_{K}^{k+2}(v), \n_K {\cdot} \nabla \Delta w)_{\dK}
+ (\Pi_{\dK}^k (\n_K{\cdot}\nabla v),  \partial_{nn}  w)_{\dK}
+  (\partial_t \interC_{K}^{k+2}(v),  \partial_{nt}  w)_{\dK}.
\end{align*}
Integrating by parts the last term on the right-hand side gives (for $d=3$)
\[
(\partial_t \interC_{K}^{k+2}(v),  \partial_{nt}  w)_{\dK}
=  \sum_{F\in \mathcal{F}_K}  \Big\{ -( \interC_{K}^{k+2}(v),  \partial_{t} (\partial_{nt}  w))_{F}
+ \sum_{E\subset \partial F} ( \interC_{K}^{k+2}(v),  \n_F {\cdot} \partial_{nt}  w)_{E}\Big\}.
\]
Owing to~\eqref{eq:def_canonical_FEM} and since $w\in \mathbb{P}_{k+2}(K)$, we infer that
\[
(\nabla^2 T\uHHO_K(v), \nabla^2 w)_{K}
=( v,  \Delta^2 w)_{K}
-(v, \n_K {\cdot} \nabla \Delta w)_{\dK}
+ ( \n_K{\cdot}\nabla v,  \partial_{nn}  w)_{\dK}
+  ( \partial_t v,  \partial_{nt}  w)_{\dK},
\]
whence we deduce from integration by parts that \eqref{eq:H2_ell_proj} holds true.
\\
(2) Proof of~\eqref{Error bound in sharp norm for elliptic projection}. Proceeding as in Lemma \ref{lem: approxmation} yields
\[ \|v - T\uHHO_K (v)\|_{\sharp, K}^2 + S_{\dK}(\widehat{\mathcal{P}}^k_K(v),\widehat{\mathcal{P}}^k_K(v))
\leq
C \| v - \interC_K^{k+2}(v)\|_{\sharp,K}^2,
\]
and we conclude owing to~\eqref{eq:inter_C_L}.
\end{proof}

We are now ready to prove Theorem~\ref{Theorem:error_L2_H1}. The proof proceeds similarly to
that of \cite[Thm.~10]{DiPEL:14} on the $L^2$-error estimate for second-order elliptic PDEs, but
requires some nontrivial adaptations using the above tools. We detail the proof of~\eqref{eq:L2err error} and only outline the minor changes needed to prove~\eqref{eq:H1err2 error}.

\begin{proof}[Proof of~\eqref{eq:L2err error}]
Let $\widehat{\mathcal{P}}_h^k:H^2(\Omega)\to \fes$ be defined such that,
for all $v\in H^2(\Omega)$,
\[
\widehat{\mathcal{P}}_h^k(v):=\big( \interC_h^{k+2}(v),(\Pi_F^{k}(\n_F{\cdot}(\nabla v)_{|F}))_{F\in\Fall} \big)\in \fes,
\]
where $\interC_h^{k+2}:H^2(\Omega)\rightarrow \mathcal{P}_{k+2}(\mesh)$ is the global interpolation operator associated with the canonical hybrid finite element.
Notice that the local components of $\widehat{\mathcal{P}}_h^k(v)$ are
$\widehat{\mathcal{P}}^k_K(v|_K)$ for all $K\in\mesh$.
Let $\hat{u}_h:=(u_\mesh,\gamma_\Fall)\in\fes$ be the discrete solution.
Let us set
\[
\hat{e}_h := \hat{u}_h - \widehat{\mathcal{P}}_h^k(u),
\]
so that, setting $\hat{e}_h:=(e_{\mesh},\theta_{\Fall})$, we have $e_K= u_K - \interC_K^{k+2}(u)$ for all $K\in\mesh$ and $\theta_F = \gamma_F - \Pi_F^k (\n_F{\cdot}\nabla u)$ for all $F\in\Fall$.
\\
(1) Let $\zeta_e$ be the adjoint solution such that $\Delta^2 \zeta_e = e_\mesh$ in $\Omega$
with the appropriate type of BC's. Owing to our regularity assumption with $\sigma=0$, we have
$\|\zeta_e\|_{H^4(\Omega)} \le C_\Omega \|e_\mesh\|_\Omega$.
Integration by parts gives
\begin{align*}
\|e_{\mesh}\|_{\Omega}^2
={}&(e_{\mesh},\Delta^2 \zeta_e)_{\Omega} = -(\nabla e_{\mesh},\nabla \Delta \zeta_e)_{\Omega}
=  -(\nabla e_{\mesh},\nabla{\cdot}\nabla^2\zeta_e)_{\Omega} \\
={}& \su \Big\{ (\nabla^2 e_{\mesh},\nabla^2 \zeta_e)_{K} - (\partial_{n}  e_{\mesh},\partial_{nn} \zeta_e)_{\dK} - (\partial_{t}  e_{\mesh},\partial_{nt} \zeta_e)_{\dK} \Big\}.
\end{align*}
The last summation on the right-hand side vanishes for both types of BC's. Therefore, we have
\begin{align*}
\|e_{\mesh}\|_{\Omega}^2
={}& \su \Big\{ (\nabla^2 e_{\mesh},\nabla^2 \zeta_e)_{K} - (\partial_{n}  e_{\mesh},\partial_{nn} \zeta_e)_{\dK}\Big\} \\
={}& \su \Big\{ (\nabla^2 e_{\mesh},\nabla^2 \zeta_e)_{K} - (\partial_{n}  e_{\mesh}- \theta_{\Fall},\partial_{nn} \zeta_e)_{\dK}\Big\} \\
={}& \su \Big\{ (\nabla^2 e_{\mesh},\nabla^2 T\uHHO_K(\zeta_e))_{K} - (\partial_{n}  e_{\mesh}- \theta_{\Fall},\partial_{nn} \zeta_e)_{\dK}\Big\},
\end{align*}
where we used that $\su(\theta_{\Fall},\partial_{nn} \zeta_e)_{\dK}=0$ for both types of BC's on the second line and the identity~\eqref{eq:H2_ell_proj} on the third line.
Since
\[ (\nabla^2 R_K(\hat{e}_K),\nabla^2 T\uHHO_K(\zeta_e))_{K}  = (\nabla^2 e_K,\nabla^2 T\uHHO_K(\zeta_e))_{K} - (\partial_{n}  e_{K}- \theta_{\dK},\partial_{nn} T\uHHO_K(\zeta_e))_{\dK},
\]
straightforward algebra using the definition of $\hat{e_h}$, that $\hat{u}_h$ solves the discrete HHO problem, and that $(f,\zeta_e)_{\Omega} = (\nabla^2 u, \nabla^2 \zeta_e)_{\Omega}$ gives
$\|e_{\mesh}\|_{\Omega}^2 = T_1+T_2-T_3$ with
\begin{align*}
T_1:= {}&  -\su \Big\{ S_{\dK}(\hat{e}_K,\widehat{\mathcal{P}}^k_K(\zeta_e))  - (\partial_{n}  e_K- \theta_{\dK},\partial_{nn} (\zeta_e-T\uHHO_K(\zeta_e)))_{\dK} \Big\}, \\
T_2:= {}&  (\nabla^2 u, \nabla^2 \zeta_e)_{\Omega}-  a_h(\widehat{\mathcal{P}}^k_h(u),\widehat{\mathcal{P}}^k_h(\zeta_e)), \\
T_3:= {}& (f, \zeta_e - \interC_{h}^{k+2}(\zeta_e))_{\Omega} = (f- \Pi_{h}^{k-2} (f), \zeta_e - \interC_{h}^{k+2}(\zeta_e))_{\Omega},
\end{align*}
where the last equality for $T_3$ follows from \eqref{cell property} (with the convention that $\Pi_{h}^{k-2} (f)=0$ for $k\in\{0,1\}$). It remains to bound the above three terms.
\\
(2) Owing to the Cauchy--Schwarz inequality and \eqref{Error bound in sharp norm for elliptic projection}, we have
$$
|T_1| \leq C \|\hat{e}_h\|_{\fes} \|\zeta_e - \interL_{h}^{k+2}(\zeta_e)\|_{\sharp,h}.
$$
Moreover, since $\|\zeta_e -\interL_{h}^{k+2}(\zeta_e) \|_{\sharp,h}\leq Ch^{\min(k+1,2)}|\zeta_e|_{H^{\min(k+3,4)}(\Omega)}$ and $|\zeta_e|_{H^{\min(k+3,4)}(\Omega)}\leq  C_{\Omega} \|e_{\mesh}\|_{\Omega}$ by our assumption on regularity pickup, we infer that
$$
|T_1| \leq Ch^{\min(k+1,2)} \|\hat{e}_h\|_{\fes} \|e_{\mesh}\|_{\Omega}.
$$
Furthermore, using the definition of $a_h$ and the identity~\eqref{eq:H2_ell_proj} yields
\begin{align*}
T_2 =&  \su \Big\{ (\nabla^2 (u-T\uHHO_K(u)), \nabla^2 (\zeta_e - T\uHHO_K(\zeta_e)))_{K} - S_{\dK}(\widehat{\mathcal{P}}^k_K(u),\widehat{\mathcal{P}}^k_K(\zeta_e)) \Big\}.
\end{align*}
Using the Cauchy--Schwarz inequality and the same arguments as for $T_1$ gives
$$
|T_2| \leq Ch^{\min(k+1,2)}  \|u -\interL_h^{k+2}(u)\|_{\sharp,h} \|e_{\mesh}\|_{\Omega} .
$$
Finally, we have
$$
|T_3|\leq \|\tilde{h}^2(f- \Pi_{h}^{k-2}(f))\|_{\Omega} \|\tilde{h}^{-2}(\zeta_e - \interC_{h}^{k+2}(\zeta_e)))\|_{\Omega} \leq  Ch^{\min(k+1,2)} \|\tilde{h}^2(f- \Pi_{h}^{k-2}(f))\|_{\Omega}  \|e_{\mesh}\|_{\Omega}.
$$
Putting everything together gives
\begin{equation}\label{bound: discrete error}
 \|e_{\mesh}\|_{\Omega} \leq Ch^{\min(k+1,2)} \Big( \|\hat{e}_h\|_{\fes} + \|u -\interL_h^{k+2}(u) \|_{\sharp,h}+ \|\tilde{h}^2(f- \Pi_{h}^{k-2}(f))\|_{\Omega}\Big).
\end{equation}

(3) Using the triangle inequality and the definition of $e_\mesh$ gives
\begin{align*}
\su \|u- R_K(\hat{u}_K)\|_{K}^2 \leq  &  C \su \Big\{ \| e_K\|_{K}^2 +  \|u_K- R_K(\hat{u}_K)\|_{K}^2 + \|u- \interC_{K}^{k+2}(u))\|_{K}^2  \Big\}.
\end{align*}
The first term on the right-hand side is bounded using~\eqref{bound: discrete error}.
The second term is bounded by the Poincar\'e inequality,~\eqref{eq:bound_uK_RuK}, and the triangle inequality giving
\begin{align*}
\|u_K-R_K(\hat{u}_K)\|_{K} \leq   Ch_K^2 \|\nabla^2 (u_K-R_K(\hat{u}_K))\|_{K} &\leq
Ch_K^2	S_{\dK}(\hat{u}_K,\hat{u}_K) \\ & \le Ch_K^2\big(S_{\dK}(\hat{e}_K,\hat{e}_K)
+ \|u-\interL_K^{k+2}(u)\|_{\sharp,K}\big).
\end{align*}
Finally, observing that $\interL_K^1$ leaves $(u-\interC_K^{k+2}(u))$ invariant,
invoking~\eqref{eq:PK} and~\eqref{eq:inter_C_L} gives
\[
\|u-\interC_K^{k+2}(u)\|_K \le Ch_K^2\|\nabla^2 (u-\interC_K^{k+2}(u))\|_K
\le Ch_K^2\|u-\interL_K^{k+2}(u)\|_{\sharp,K}.
\]
Putting everything together proves~\eqref{eq:L2err error} since $h^2 \le Ch^{min(k+1,2)}$.
\end{proof}

\begin{proof}[Proof of~\eqref{eq:H1err2 error}]
The only salient change with respect to the previous proof is the definition of the adjoint solution which now satisfies $\Delta^2 \zeta_e = \Delta e_\mesh \in H^{-1}(\Omega)$ with the appropriate type of BC's (recall that $e_\mesh\in H^1_0(\Omega)$ by definition). Proceeding as above and invoking the regularity pickup assumption~\eqref{eq:pickup} with $\sigma=-1$ then gives
\begin{align*}
\|\nabla e_\mesh\|_\Omega^2 = \langle \Delta e_\mesh, e_{\mesh} \rangle_{H^{-1},H^{1}_0}   &  \leq Ch \Big( \|\hat{e}_h\|_{\fes} + \|u -\interL_h^{k+2}(u) \|_{\sharp,h}+ \|\tilde{h}^2(f- \Pi_{h}^{k-2}(f))\|_{\Omega}\Big)|\zeta_e|_{H^{3}(\Omega)} \\
 & \leq Ch \Big( \|\hat{e}_h\|_{\fes} + \|u -\interL_h^{k+2}(u) \|_{\sharp,h}+ \|\tilde{h}^2(f- \Pi_{h}^{k-2}(f))\|_{\Omega}\Big)\|\Delta e_\mesh\|_{H^{-1}(\Omega)}.
\end{align*}
The assertion finally follows from $\|\Delta e_\mesh\|_{H^{-1}(\Omega)} = \|\nabla e_\mesh\|_\Omega$.
\end{proof}

\begin{remark}[Variants and decay rates]
Proceeding as in Remark~\ref{rem:variant_H2} shows that the bounds \eqref{eq:L2err error} and \eqref{eq:H1err2 error} hold true also with the left-hand side replaced by $\su\|u-u_K\|_K^2$ and $\su\|\nabla(u-u_K)\|_K^2$, respectively. Concerning the right-hand side of \eqref{eq:L2err error} and \eqref{eq:H1err2 error}, the first term is controlled by the second term owing to Theorem~\ref{Theorem:error_II}, so that both terms converge at rate $\mathcal{O}(h^{2(k+1)})$. Moreover, the last term converges at rate $\mathcal{O}(h^{2\max(2,k+1)})$ (i.e., at the same rate for $k\ge1$ and at one order faster for $k=0$). Taking into account the scaling factor in front of the parenthesis and taking the square root, we infer that the $L^2$-error converges at rate $\mathcal{O}(h^2)$ for $k=0$, which is suboptimal by one order in $h$, whereas it converges at the optimal rate $\mathcal{O}(h^{k+3})$ for all $k\ge1$. Finally, the $H^1$-error converges at the optimal rate $\mathcal{O}(h^{k+2})$ for all $k\ge0$.
\end{remark}

\section{Numerical examples} \label{sec: numerical examples}

In this section, we present numerical examples illustrating our theoretical results on the convergence of the $C^0$-HHO method. We also compare the numerical performance of the $C^0$-HHO method to other (classical) $C^0$-conforming methods from the literature.  All the computations were run with \texttt{Matlab R2021b} on the Cleps platform at INRIA Paris using 12 cores (all offering the same computational performance in terms of RAM and frequency), and all the linear systems after static condensation (if applicable) are solved using the \verb*|backslash| function (invoking Cholesky's factorization).

\subsection{Convergence rates for the $C^0$-HHO method}
\label{sec:res_poly}

We select $f$ and type (I) BC's on $\Omega:=(0,1)^2$ so that the exact solution is
$$
u(x,y) := \sin(\pi x)^2  \sin(\pi y)^2+e^{-(x-0.5)^2-(y-0.5)^2}.
$$
We employ the polynomial degrees $k\in\{0,\ldots,4\}$ and a sequence of successively refined triangular meshes consisting of $\{32,128,512,2048,8192,32768\}$ cells. Despite an $hp$-error analysis falls beyond the present scope, we weigh the stabilization terms in \eqref{def: stabilisation} by replacing $h_K^{-1}$ by $(k+1)^2 h_K^{-1}$ for all $K\in\mesh$.

Let us first verify the convergence rates obtained with the $C^0$-HHO method
with $k\in\{0,1,2,3,4\}$. We consider a sequence of successively refined triangular meshes.
We measure errors in the (broken) $H^2$-seminorm,  in the $H^1$-seminorm, in the $L^2$-norm, and in the stabilization seminorm. The first three errors are evaluated using the reconstruction of the HHO solution cellwise.
The errors are reported in Figure \ref{ex1:h-refine}, and the rates are reported in Table \ref{ex1: table: compute the convergence rate} as a function of
$\mathrm{DoFs}^{1/2}$, where $\mathrm{DoFs}$ denotes the total number of globally coupled
discrete unknowns (that is, the face unknowns in addition to the cell unknowns except the bubble functions). We observe that the $H^2$-error,  $H^1$-error, and stabilization error converge at the optimal rates $O(h^{k+1})$, $O(h^{k+2})$, and $O(h^{k+1})$, respectively,  as expected.
The $L^2$-error converges at the optimal rate $O(h^{k+3})$, except for $k=0$ where the rate
is only $O(h^2)$. All these rates are consistent with the analysis presented in the previous section.
We also notice that some errors and decay rates are not reported on the finest meshes for the higher polynomial degrees. The reason is that a stagnation of the error at levels around $10^{-8}$ is observed owing to the poor conditioning of the linear system (see below). Notice that stagnation does not affect the stabilization error.

\begin{figure}[t]
	\centering
	\includegraphics[scale=0.32]{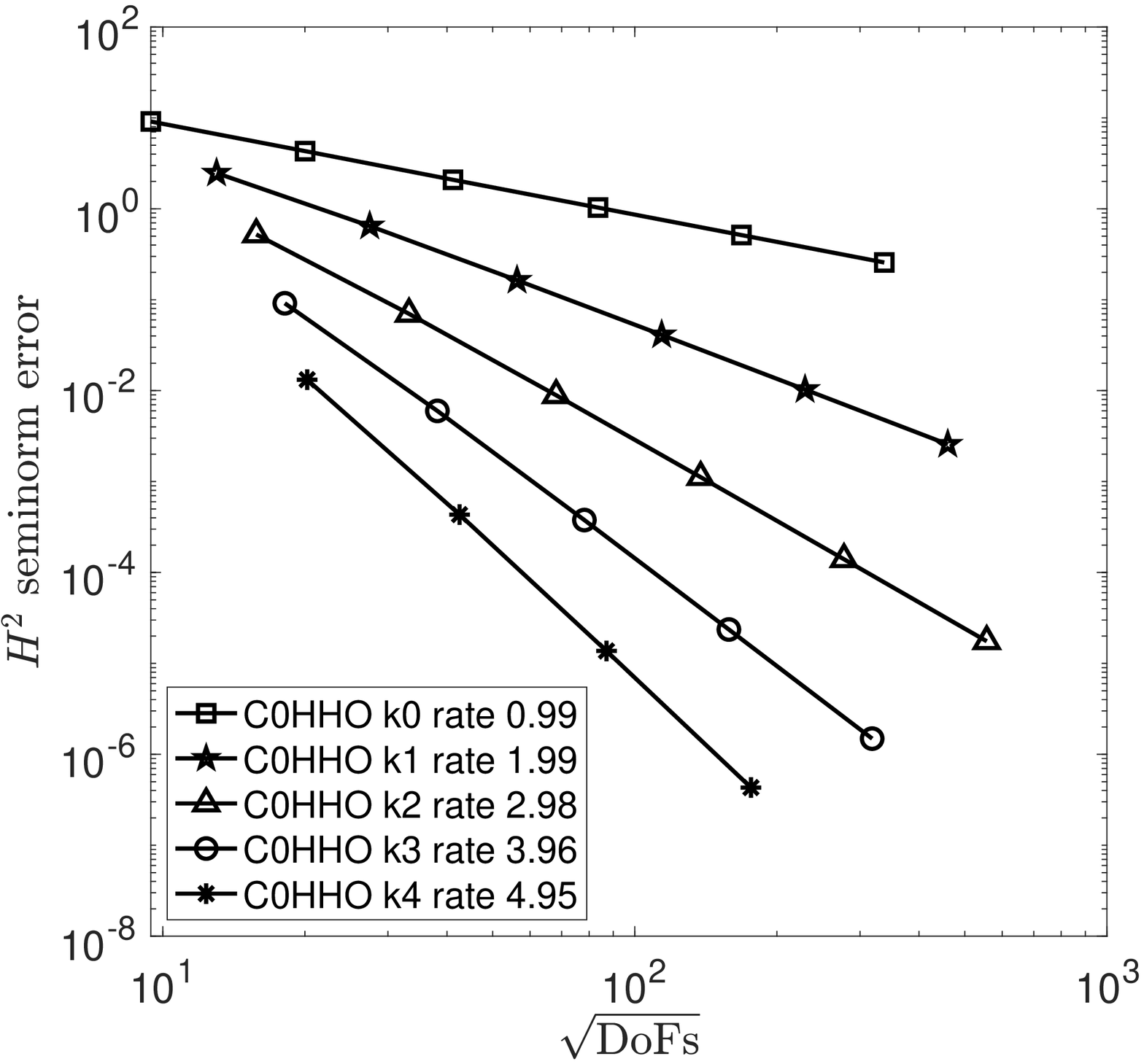}	\includegraphics[scale=0.32]{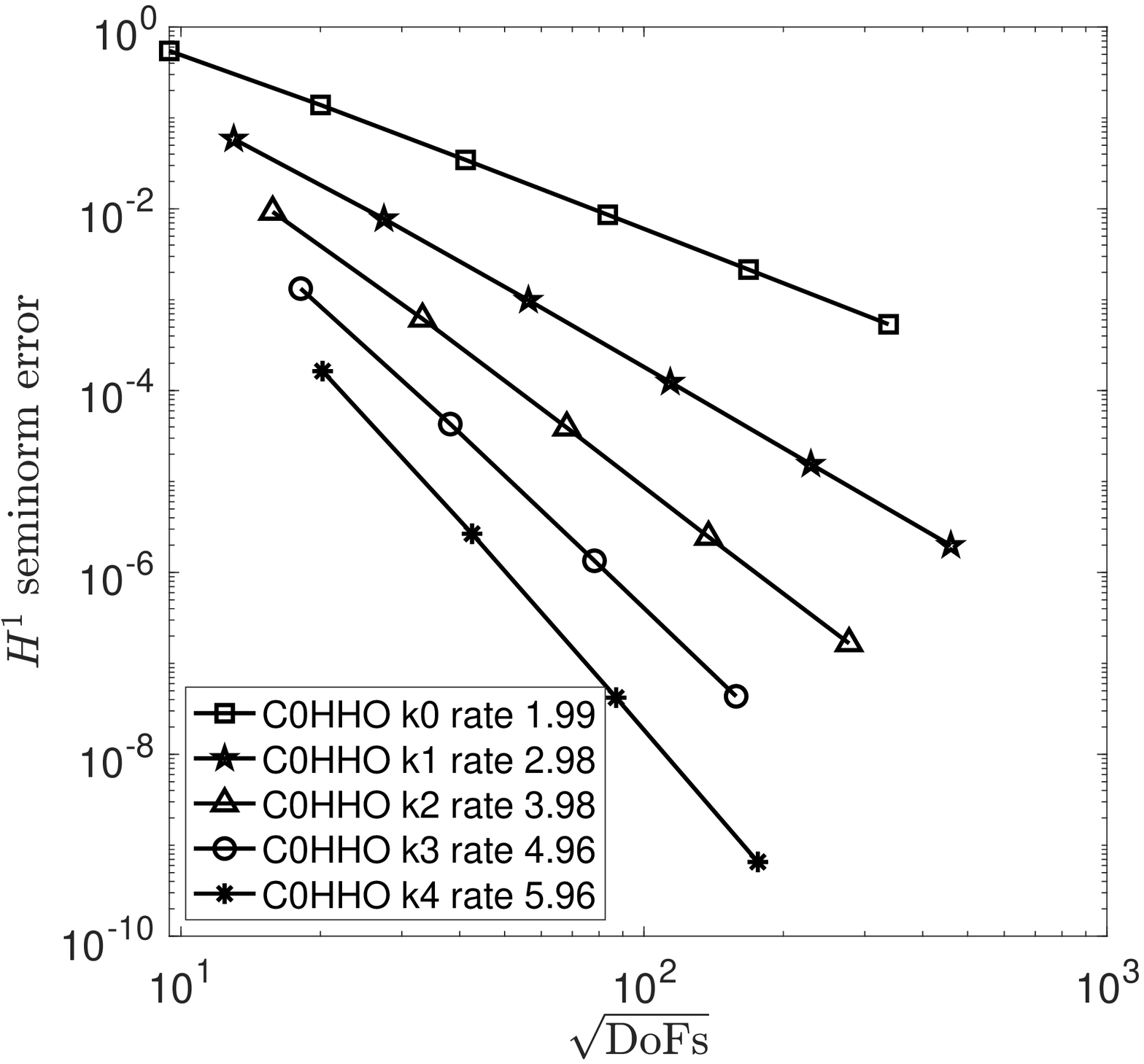}\\
	\includegraphics[scale=0.32]{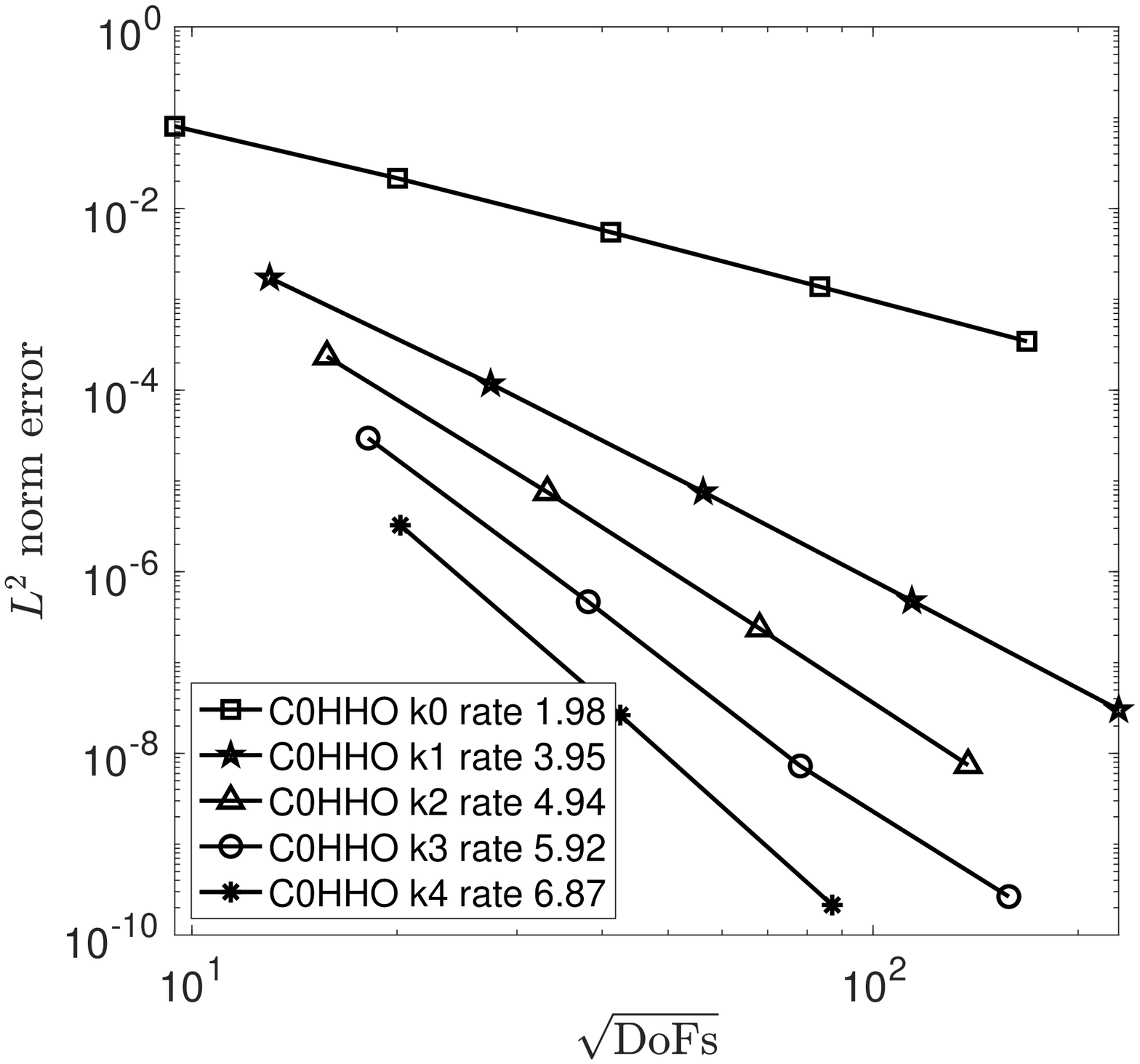}
	\includegraphics[scale=0.32]{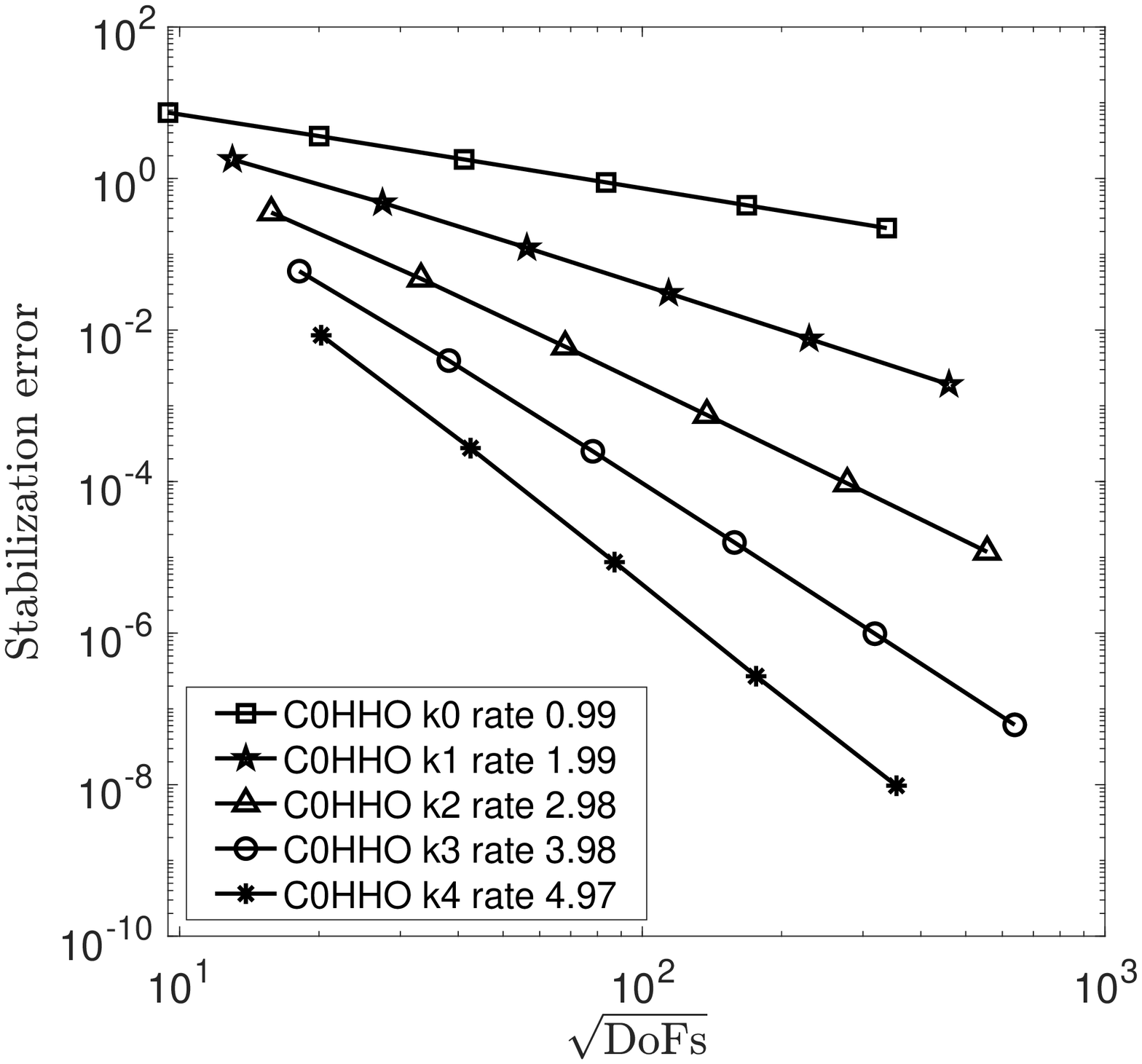}
	\caption{
		\label{ex1:h-refine}
		Convergence of $C^0$-HHO in $H^2$-, $H^1$-, $L^2$-(semi)norms, and stabilization seminorm on triangular meshes.}
\end{figure}

\begin{table}[!htb]
		\centering
		\begin{tabular}{|c|c|c|c|c|c|c|c|c|c|}
			\hline
			$H^2$ error& \multicolumn{2}{|c|}{$k=0$}    & \multicolumn{2}{|c|}{$k=1$} 	& \multicolumn{2}{|c|}{$k=2$}  & \multicolumn{2}{|c|}{$k=3$} 		   \\
			\hline
			$\# \mesh$
			& $\#$ \text{DoFs}  & \text{ rate } & $\#$ \text{DoFs} & \text{ rate } & $\#$ \text{DoFs}  & \text{ rate } & $\#$ \text{DoFs} & \text{ rate }\\
			\hline
			32 & 137  & --- & 281 & --- & 457 & --- & 665 & --- \\
			\hline
			128 & 497   & 0.99& 1041&1.62 & 1713 & 2.69 & 2513 &  3.66\\
			\hline
			512 &  1889  & 1.00& 4001  & 1.91 & 6625 & 2.87& 9761  & 3.86\\
			\hline
			2048 & 7361  & 0.99 & 15681& 1.96 & 26049  & 2.94 & 38465& 3.93 \\
			\hline
			8192 & 29057  & 0.98& 62081 & 1.98 & 103197 & 2.97 & 152705 & 3.96\\
			\hline
			32768 & 115457  & 0.99& 247041  & 1.99 & 411393 & 2.99 & 608513 & -\\
			\hline
\hline
			$H^1$ error& \multicolumn{2}{|c|}{$k=0$}    & \multicolumn{2}{|c|}{$k=1$} 	& \multicolumn{2}{|c|}{$k=2$}  & \multicolumn{2}{|c|}{$k=3$} 		   \\
			\hline
			$\# \mesh$
			& $\#$ \text{DoFs}  & \text{ rate } & $\#$ \text{DoFs} & \text{ rate } & $\#$ \text{DoFs}  & \text{ rate } & $\#$ \text{DoFs} &\text{ rate }\\
			\hline
			32 & 137  & --- & 281 & --- & 457 & --- & 665 & --- \\
			\hline
			128 & 497   & 1.82& 1041&2.71 & 1713 & 3.63 & 2513 &  4.66\\
			\hline
			512 &  1889  & 1.93& 4001  & 2.89 & 6625 & 3.87& 9761  & 4.85\\
			\hline
			2048 & 7361  & 1.96 & 15681& 2.95 & 26049  & 3.94 & 38465& 4.96 \\
			\hline
			8192 & 29057  & 1.98& 62081 & 2.97 & 103197 & 3.98 & 152705 & -\\
			\hline
			32768 & 115457  & 1.99& 247041  & 2.98 & 411393 & - & 608513 & -\\
			\hline
			\hline
			$L^2$ error & \multicolumn{2}{|c|}{$k=0$}    & \multicolumn{2}{|c|}{$k=1$} 	& \multicolumn{2}{|c|}{$k=2$}  & \multicolumn{2}{|c|}{$k=3$} 		   \\
			\hline
			$\# \mesh$
			& $\#$ \text{DoFs}  & \text{ rate } & $\#$ \text{DoFs} & \text{ rate } & $\#$ \text{DoFs}  & \text{ rate } & $\#$ \text{DoFs} & \text{ rate }\\
			\hline
			32 & 137  & --- & 281 & --- & 457 & --- & 665 & --- \\
			\hline
			128 & 497   & 1.76& 1041&3.60 & 1713 & 4.62 & 2513 &  5.58\\
			\hline
			512 &  1889  & 1.90& 4001  &3.82  & 6625 & 4.82& 9761  & 5.92\\
			\hline
			2048 & 7361  & 1.95 & 15681& 3.92 & 26049  & 4.94 & 38465& 5.82 \\
			\hline
			8192 & 29057  & 1.97& 62081 & 3.95 & 103197 & - & 152705 & -\\
			\hline
			32768 & 115457  & 1.99& 247041  & - & 411393 & - & 608513 & -\\
			\hline
\hline
			Stab error & \multicolumn{2}{|c|}{$k=0$}    & \multicolumn{2}{|c|}{$k=1$} 	& \multicolumn{2}{|c|}{$k=2$}  & \multicolumn{2}{|c|}{$k=3$} 		   \\
			\hline
			$\# \mesh$
			& $\#$ \text{DoFs}  & \text{ rate } & $\#$ \text{DoFs} & \text{ rate } & $\#$ \text{DoFs}  & \text{ rate } & $\#$ \text{DoFs} & \text{ rate }\\
			\hline
			32 & 137  & --- & 281 & --- & 457 & --- & 665 & --- \\
			\hline
			128 & 497   & 0.95& 1041&1.72 & 1713 & 2.70 & 2513 &  3.66\\
			\hline
			512 &  1889  & 0.98& 4001  &1.90 & 6625 & 2.88& 9761  & 3.85\\
			\hline
			2048 & 7361  & 0.99 & 15681& 1.95 & 26049  & 2.95 & 38465& 3.93 \\
			\hline
			8192 & 29057  & 0.99& 62081 & 1.98 & 103197 & 2.98 & 152705 & 3.97\\
			\hline
			32768 & 115457  & 0.99& 247041  & 1.99 & 411393 & 2.99 & 608513 & 3.99\\
			\hline
		\end{tabular}
	\caption{Convergence rates of $C^0$-HHO in $H^2$-, $H^1$-, $L^2$-(semi)norms, and stabilization seminorm on triangular meshes.}\label{ex1: table: compute the convergence rate}
\end{table}

\begin{table}[!htb]
		\centering
		\begin{tabular}{|c|c|c|c|c|c|c|c|}
			\hline
			Cond No  & \multicolumn{2}{|c|}{$k=2$}  & \multicolumn{2}{|c|}{$k=3$} 	&	  \multicolumn{2}{|c|}{$k=4$}  \\
			\hline
			$\# \mesh$
			 &  \text{condensed}  & \text{full} &  \text{condensed} & \text{full}& \text{condensed} & \text{full}\\
			\hline
			32  & 2.99e+04& 8.74e+04& 8.56e+04 & 3.93e+05& 2.02e+05 & 1.53e+06 \\
			\hline
			128& 2.96e+05 & 9.33e+05&  8.82e+05 &  4.41e+06& 2.13e+06 &  1.77e+07\\
			\hline
			512 & 3.41e+06 & 1.14e+07& 1.05e+07 & 5.64e+07& 2.59e+07 &  2.34e+08\\
			\hline
			2048 &  4.96e+07  & 1.72e+08 & 1.55e+08&8.65e+08&  3.89e+08&  3.62e+09\\
			\hline
			8192 &  7.42e+08 & 2.63e+09 & 2.36e+09 & 1.34e+10& 5.95e+09&  5.68e+10\\
			\hline
			32768 & 1.13e+10 & 4.09e+10 & 3.63e+10 & 2.11e+11& 9.23e+10 &  8.97e+11\\
			\hline
		\end{tabular}
	\caption{Condition number of condensed  and  full linear system for $C^0$-HHO on triangular meshes.}\label{ex1: table: compute the condition number}
\end{table}

Table \ref{ex1: table: compute the condition number} reports the condition number of the full and condensed linear systems arising from the $C^0$-HHO methods with $k\in\{2,3,4\}$ on the same meshes as above. For fixed $k$, all the condition numbers scale as $\mathcal{O}(h^{-4})$, as expected. In addition, the condition number of the condensed linear system is always smaller than that of the full linear system. The ratio is about $4$, $6$, $9$ for $k=2$, $3$, $4$, respectively. Furthermore, the condition number of the condensed linear system is above $10^{10}$ on the finest mesh for $k\in\{2,3\}$ and on the two finest meshes for $k=4$, thereby causing an error stagnation at about $10^{-8}$ with our current implementation. Finally, we mention that the above condition numbers correspond to nodal basis functions based on the Fekete points. We also tested modal basis functions, but obtained even larger condition numbers (by a factor of about $2$, $4$, $8$ for $k=2$, $3$, $4$, respectively).

\subsection{Comparison with $C^0$-IPDG, HHO, DG, and Morley and HCT FEM}

In this section, we compare the computational performance of $C^0$-HHO with the $C^0$-IPDG,  HHO, dG,  Morley, and HCT methods on triangular meshes. We point out that we implement the jump of the normal gradient for dG and $C^0$-IPDG methods using the weight $n_{\partial}(k+1)^2 h_F^{-1}$  for all $F\in\Fall$, with $n_{\partial} = 4$. In addition, we implement the jump of the trace term for dG methods using the weight $n_{\partial}(k+1)^6 h_F^{-3}$ for all $F\in\Fall$.

We first compare the $C^0$-HHO method to the dG, HHO (more precisely, the so-called HHO(A) method from \cite{DongErn2021biharmonic}), and $C^0$-IPDG methods. To put the $C^0$-HHO and the other methods on a fair comparison basis, we compare the $C^0$-HHO and HHO methods with face polynomial degree $k\ge0$ to the dG and $C^0$-IPDG methods with cell polynomial degree $\ell=k+2$, so that all the methods deliver the same decay rates on the $H^2$-error. Moreover, all the methods are assembled by using affine geometric mappings from a reference triangle. Interestingly, we point out that using affine geometric mappings can accelerate the assembling procedure about 4 to 8 times for DG and HHO methods in contrast with the use of a physical basis in each mesh cell, see \cite[Sec. 6]{DGpolybook} for a more detailed discussion.

A comparison of total DoFs, assembling time (including
static condensation if applicable), and condition number of the linear system (again after static condensation if applicable) is presented
in Table~\ref{ex1:table Comparison for time for various methods} for the four discretization methods. We consider a triangular mesh with $32768$ cells and let the polynomial degree vary from
$k=0$ to $k=3$. Concerning assembling time (the reported values are meaningful up to 5-10\%), we observe that for the lowest-order cases ($k\in\{0,1\}$), $C^0$-HHO methods are the most effective, whereas for the higher-order cases ($k\in\{2,3\}$), $C^0$-IPDG methods take slightly less time. Another interesting observation is that $C^0$-HHO methods spend about $20\%$ to $30\%$ less time for assembling the linear system compared to (fully discontinuous) HHO methods and dG methods.
Concerning the condition number,
the first observation is that the condition number for $C^0$-HHO and (fully discontinuous) HHO methods is comparable although the latter is about 10$\%$ to 30$\%$ larger. The second  observation is that the condition number of $C^0$-HHO methods is (significantly) smaller than the one produced by $C^0$-IPDG and (fully discontinuous) dG methods. One reason for this difference is the elimination of the higher-order bubble functions from the cell unknowns in $C^0$-HHO methods. We also mention that the condition number of the $C^0$-IPDG and dG methods are about 4 times smaller if one chooses $n_{\partial} =1 $ instead of $n_{\partial} =4$, although the former value is below the minimal coercivity threshold predicted by the theory.
\begin{table}[!htb]
	\centering
		\begin{tabular}{|c|c|c|c||c|c|c|c|}
			\hline
			\multicolumn{4}{|c||}{$C^0$-HHO}
			&
			\multicolumn{4}{|c|}{$C^0$-IPDG}  \\
			\hline
			\hline
			order  & kDoFs & assembling (s) &   Cond No & order & kDoFs  & assembling  (s)&  Ratio   \\
			\hline
			$k=0$ & $114$  & $75$ &   2.27e+08 &  $\ell=2$ & $66$ & $ 112$ &   67.4   \\
			\hline
			$k=1$ & $212$  & $197$ & 2.50e+09 & $\ell=3$& $148$ & $238$ & 168.0  \\
			\hline
			$k=2$ &  $310$ & $ 507$ &  1.13e+10  & $\ell=4$ & $263$  & $498$ & 369.0  \\
			\hline
			$k=3$ & $407$ & $1091$ &  3.63e+10 & $\ell=5$ & $  411$ & $ 1061$ &   597.8 \\
			\hline
			\hline
			\multicolumn{4}{|c||}{HHO}
			&
			\multicolumn{4}{|c|}{dG}  \\
			\hline
			\hline
			order  & kDoFs  & assembling (s) &  Ratio  & order & kDoFs  & assembling  (s)& Ratio \\
			\hline
			$k=0$ & $147$   & $118$ & 1.3  &  $\ell=2$ & $197$ & $154$ &  23.1 \\
			\hline
			$k=1$ & $244$ & $306$ &1.2&  $\ell=3$& $328$ & $330$ &  70.0  \\
			\hline
			$k=2$ &  $342$  & $666$ & 1.1   & $\ell=4$ & $492$  & $718$ & 125.7 \\
			\hline
			$k=3$ & $440$ & $1284$ & 1.1& $\ell=5$ & $  688$ & $ 1438$ &  226.2\\
			\hline
		\end{tabular}
	\caption{Comparison of total DoFs, assembling time, and condition number for the $C^0$-HHO, $C^0$-IPDG, HHO and dG methods. The polynomial degree is chosen so that all the methods deliver the same decay rates on the $H^2$-error. }
	\label{ex1:table Comparison for time for various methods}
\end{table}

Finally, we compare in Table~\ref{ex2:Comparison for DoFs}
the $C^0$-HHO methods to the Morley and HCT methods on a 
a triangulation composed of 32768 cells. The first observation is that the  Morley element only takes about 65$\%$ of the assembling time of $C^0$-HHO methods with $k=0$, but the condition number of the $C^0$-HHO method is about half of that obtained with the Morley element. The second observation is that the $C^0$-HHO method with $k=1$ takes about 60$\%$ of the assembling time of the HCT element and, at the same time, delivers a condition number which is 5 times smaller than that obtained with the HCT element. One possible reason for the discrepancy in assembly time is that HCT methods cannot be constructed by using affine geometric maps and that these methods lead to larger stencils compared to $C^0$-HHO methods.

\begin{table}[htb]
		\centering
		\begin{tabular}{||c|c|c|c||c|c|c|c||}
			\hline
			$k=0$& kDoFs & assembling (s)&  Cond No
			& $k=1$ & kDoFs & assembling (s)& Cond No \\
			\hline
			Morley & $65$ & $42$ & 4.79e+08& HCT & $97$ & $ 354$ & 1.20e+10\\
			\hline
			$C^0$-HHO & $114$  & $65$ &  2.27e+08  &  $C^0$-HHO & $212$  & $197$ &  2.50e+09\\
			\hline
		\end{tabular}
	\caption{Comparison of total DoFs, assembling time, and condition number for the $C^0$-HHO, Morley, and HCT methods. The polynomial degree is chosen so that the two methods in the same column deliver the same decay rates on the $H^2$-error.
		Triangular mesh composed of $32768$ cells, $ 49408$ edges, and $16641$ vertices.}
	\label{ex2:Comparison for DoFs}
\end{table}

\subsection{Nonsmooth problem with type (II) BC's}\label{sec:res_singular}

We select $f$ and the boundary conditions on $\Omega:=(-1,1)^2$ so that the exact solution to \eqref{pde} is
\begin{equation}
	u(x,y)
	: = \left\{
	\begin{array}{ll}
		\sin(\pi y)\Big(\frac{x^2}{4 \pi} + \frac{1}{8\pi^3}(\cos(2\pi x)-1)\Big) &\quad (-1,0]\times (-1,1)
		\vspace{2mm}	\\
		\sin(\pi y)\Big(\frac{x}{4 \pi^2} - \frac{1}{8\pi^3} \sin(2\pi x)\Big) &  \quad (0,1)\times (-1,1).
	\end{array}
	\right.
\end{equation}
It can be checked that the solution satisfies $u\in H^{3.5-\epsilon}$ with $\epsilon>0$ arbitrarily close to zero. More precisely, the fourth-order derivative $\partial_{xxxx} u$ contains a Dirac measure supported on the line $\{x=0\}$. In addition, the forcing function satisfies $f\in H^{-0.5-\epsilon}(\Omega)$. Thus, the regularity of $u$ and $f$ are compatible with Assumption~\ref{ass:regularity}.

\begin{figure}[htb]
	\centering
	\includegraphics[scale=0.32]{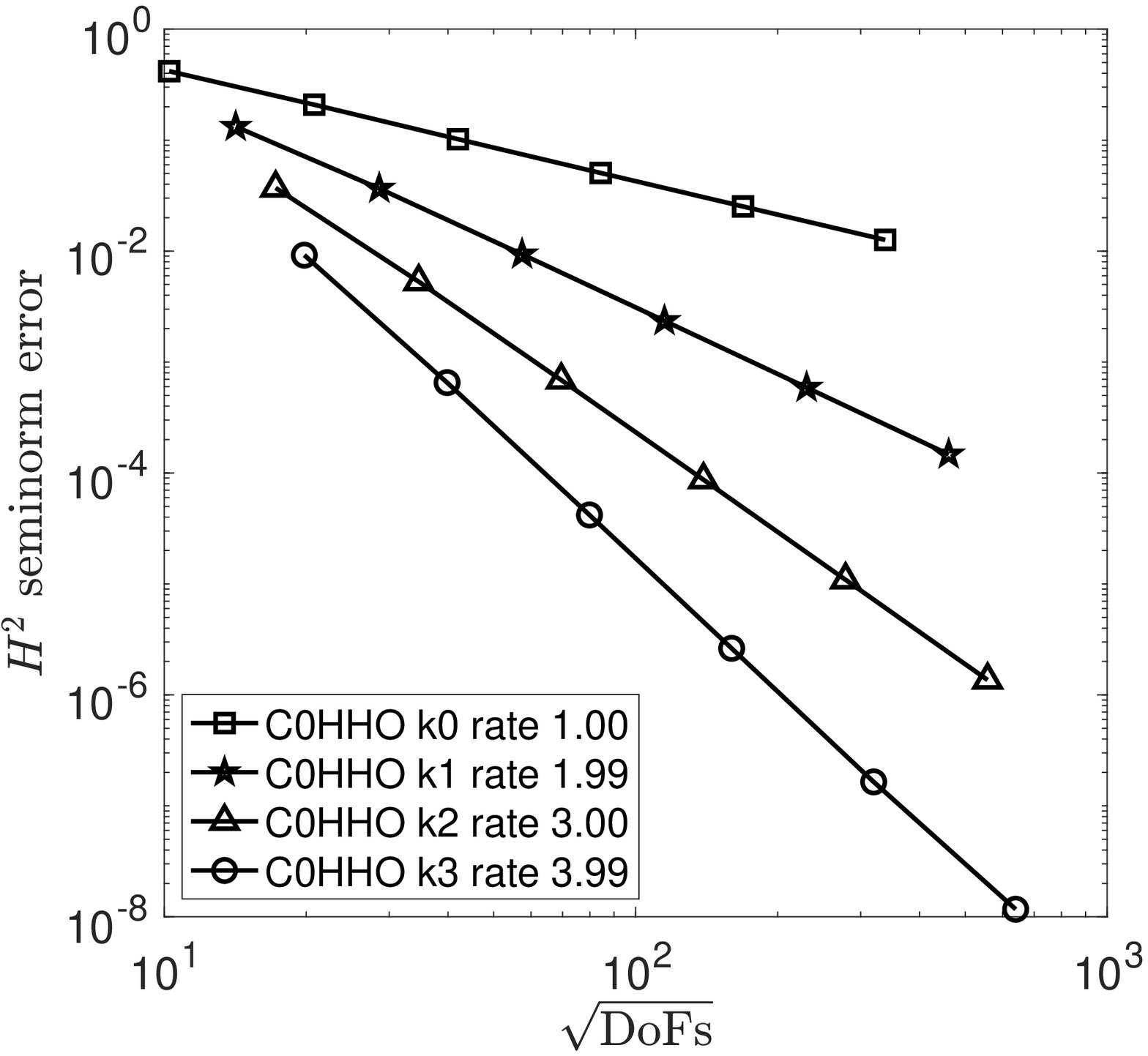}
	\includegraphics[scale=0.32]{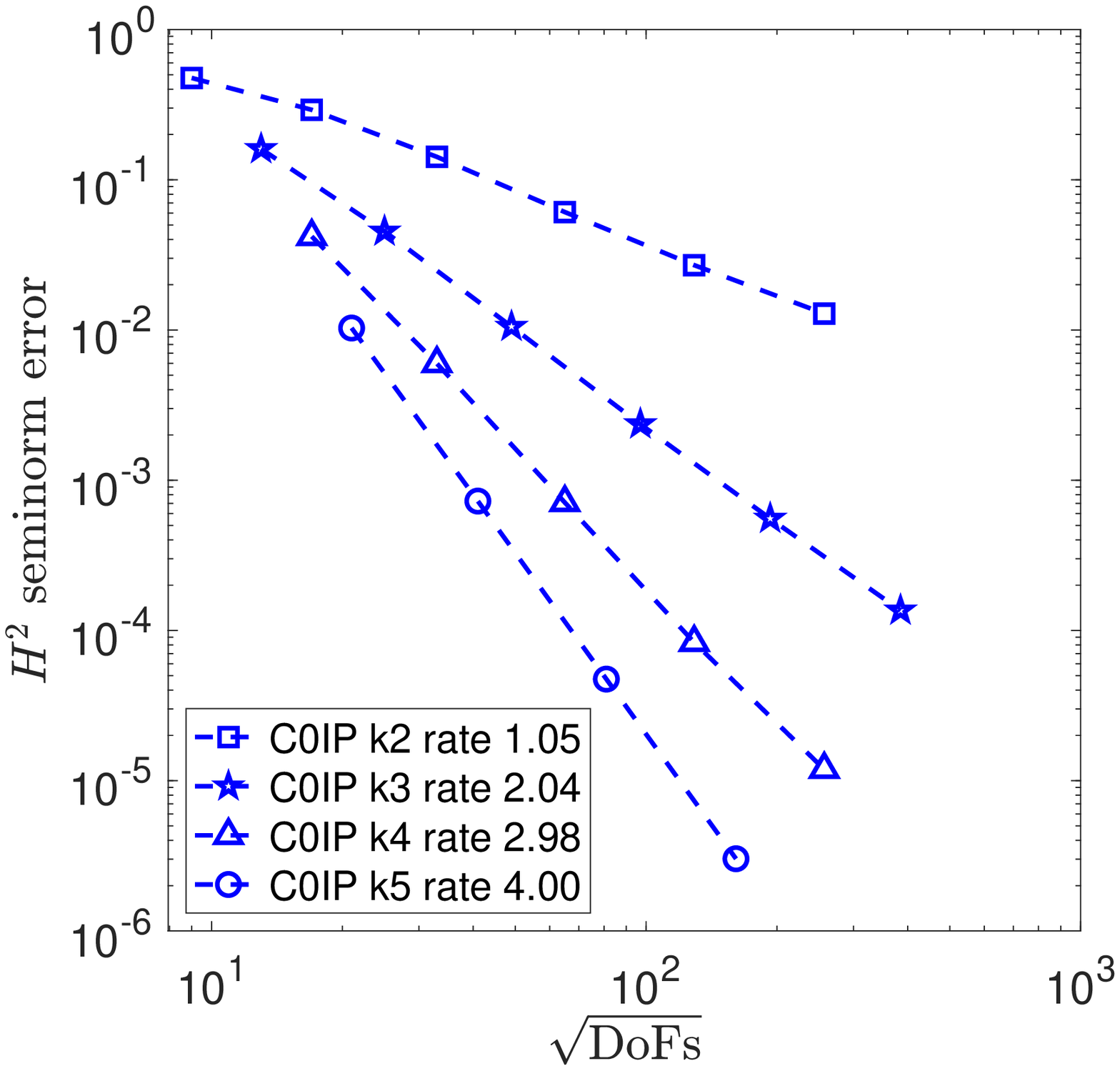}
	\caption{
		\label{ex2:h_refine_aligned_meshes}
		Convergence of $C^0$-HHO and $C^0$-IPDG methods in $H^2$-seminorms on triangular meshes with $k=0,1,2,3$.}
\end{figure}

The error decay rates depend on whether one considers meshes aligned with the line $\{x=0\}$ that supports the singularity. Considering first aligned meshes composed of $\{32,128,512,2048,8192,32768\}$ triangular cells, Figure \ref{ex2:h_refine_aligned_meshes} shows that, as expected,
the $C^0$-HHO and $C^0$-IPDG methods deliver optimal convergence rates.
Considering now non-aligned meshes composed of $\{50,162,578,$ $2178,8450,33282,132098\}$ triangular cells (see Figure \ref{ex2:mesh_figure_nonaligned} for an illustration), we notice from Figure~\ref{ex2:h_refine_nonaligned_meshes} that both
$C^0$-HHO (with $k=0$) and $C^0$-IPDG (with $\ell=k+2=2$) deliver the optimal decay rate $\mathcal{O}(h)$ for the error, whereas the decay rate levels off, as expected, at $\mathcal{O}(h^{1.5})$ in the higher-order case. This result is again in agreement with the above theoretical results.

\begin{figure}[htb]
	\centering
	\includegraphics[scale=0.3]{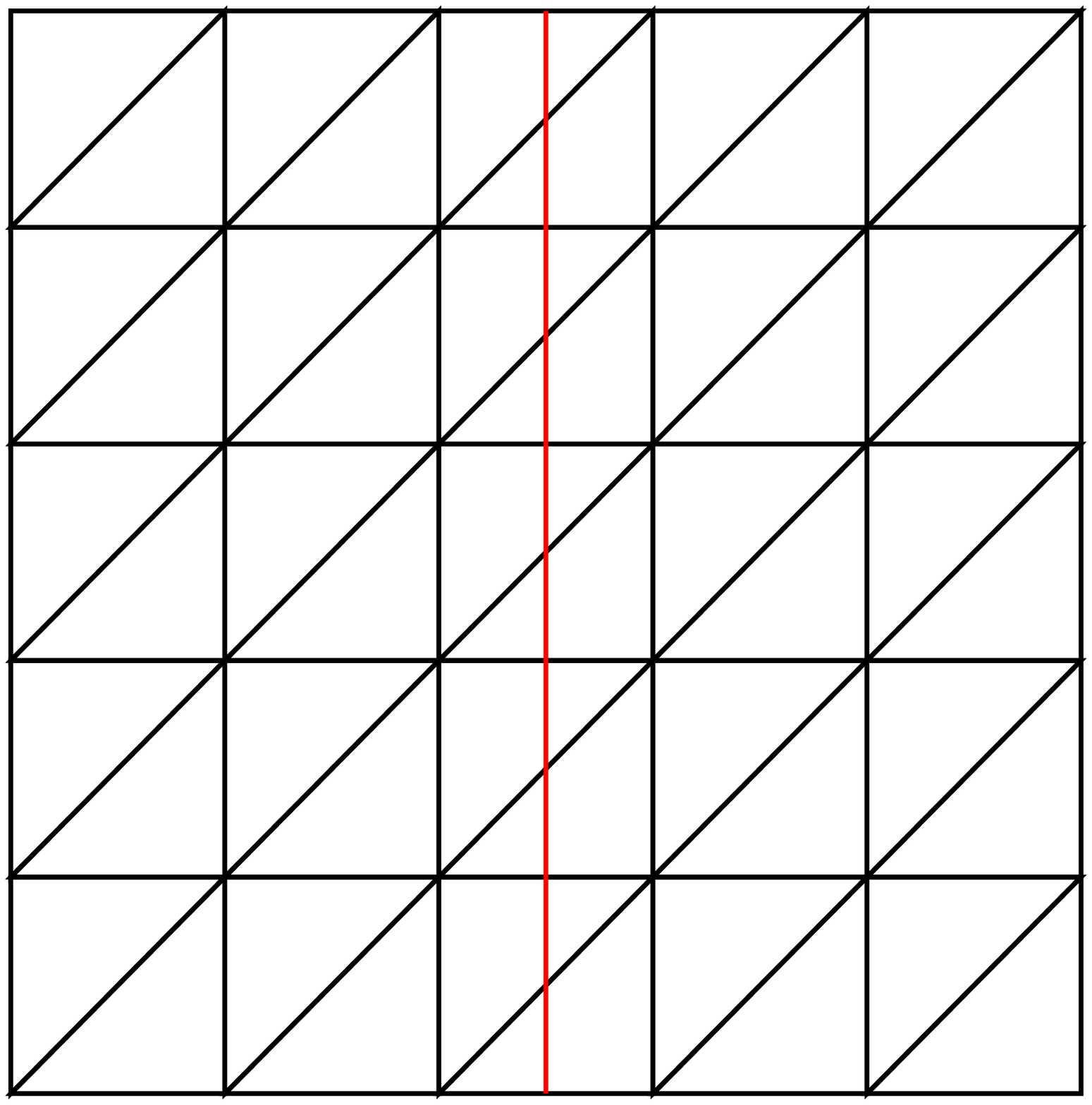}
	\includegraphics[scale=0.3]{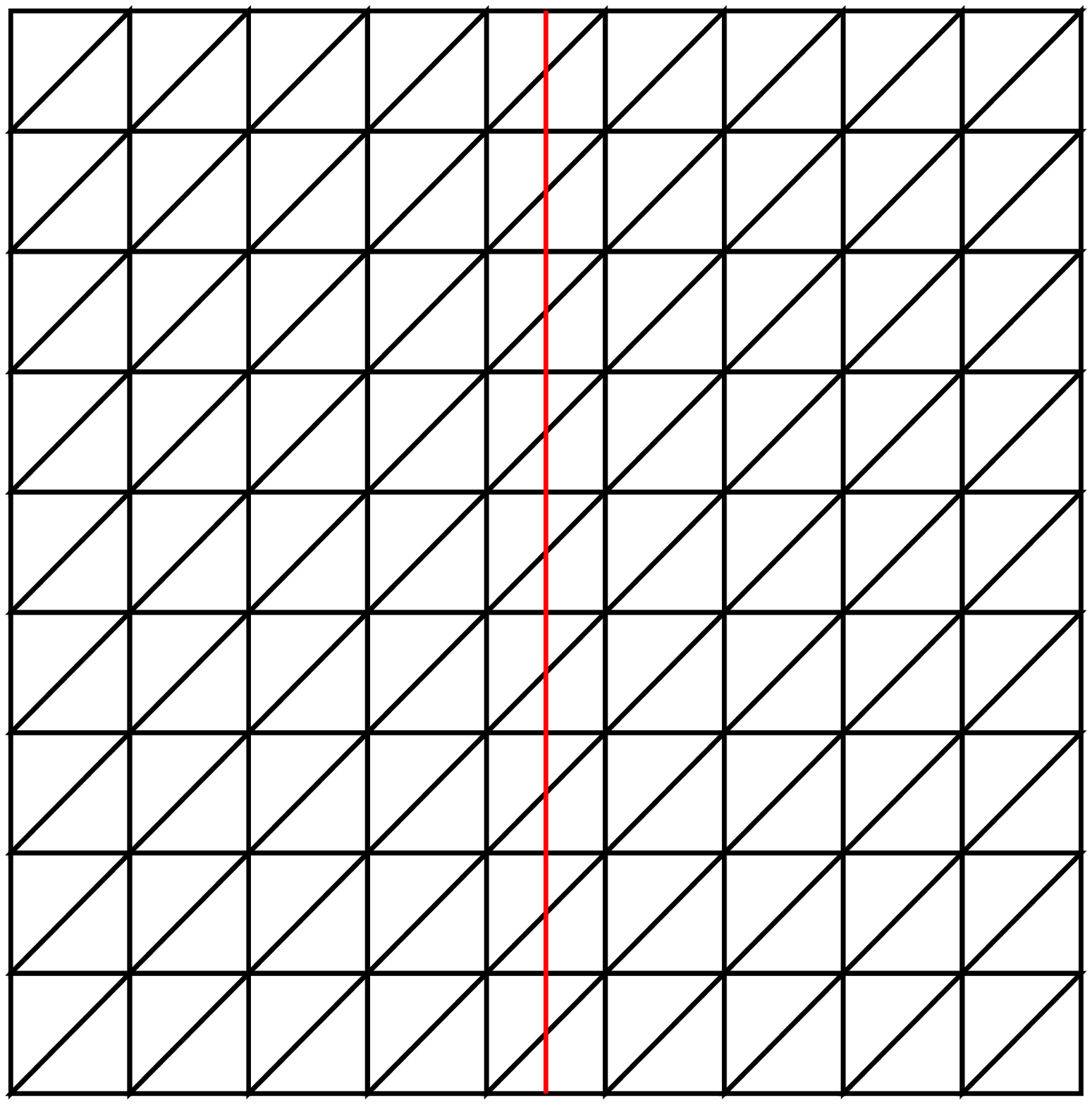}
	\caption{Two examples of triangular (nonaligned) meshes with $50$ (left) and $162$ (right) triangles.}\label{ex2:mesh_figure_nonaligned}
\end{figure}
\begin{figure}[htb]
	\centering
	\includegraphics[scale=0.32]{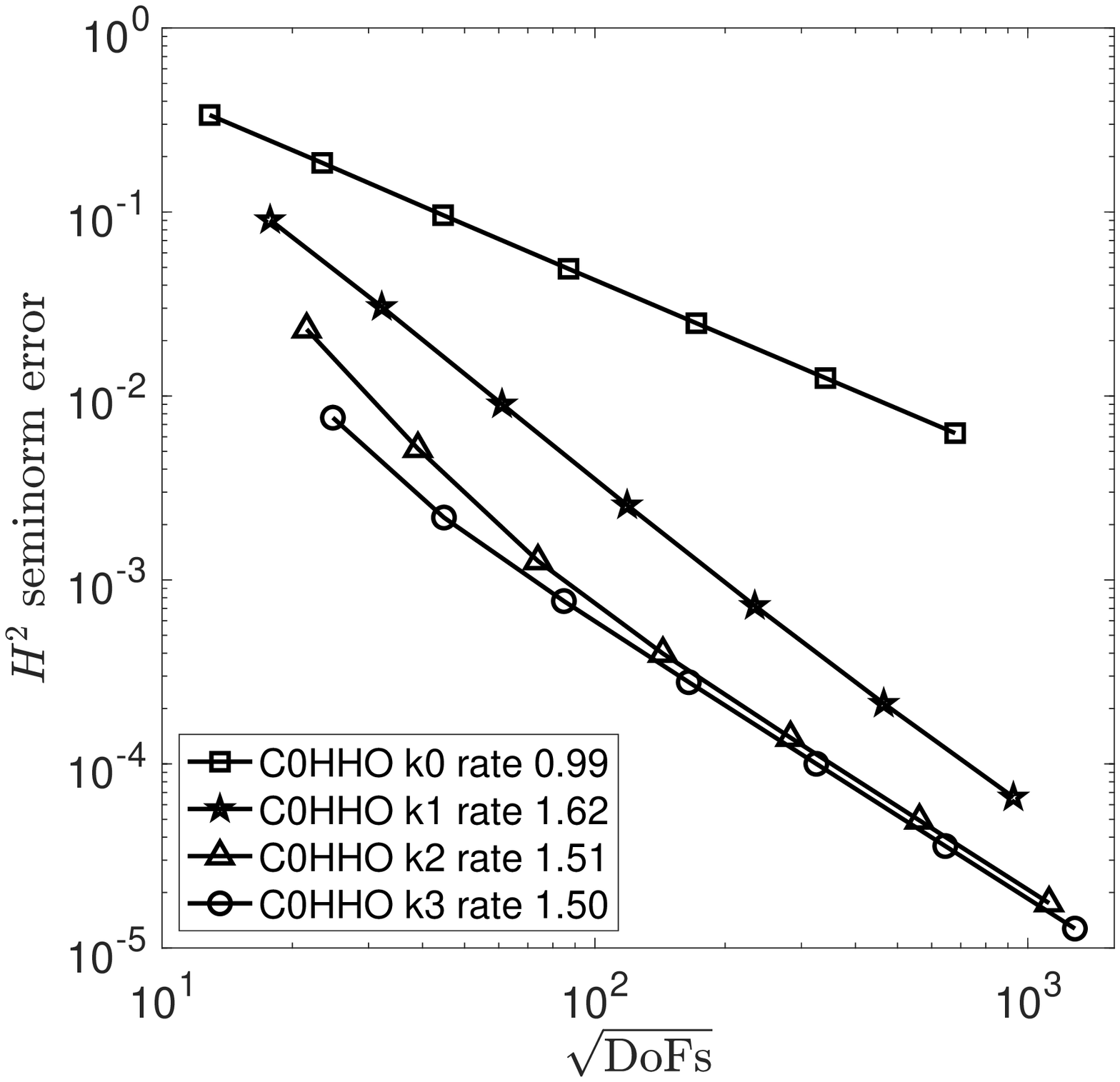}
	\includegraphics[scale=0.32]{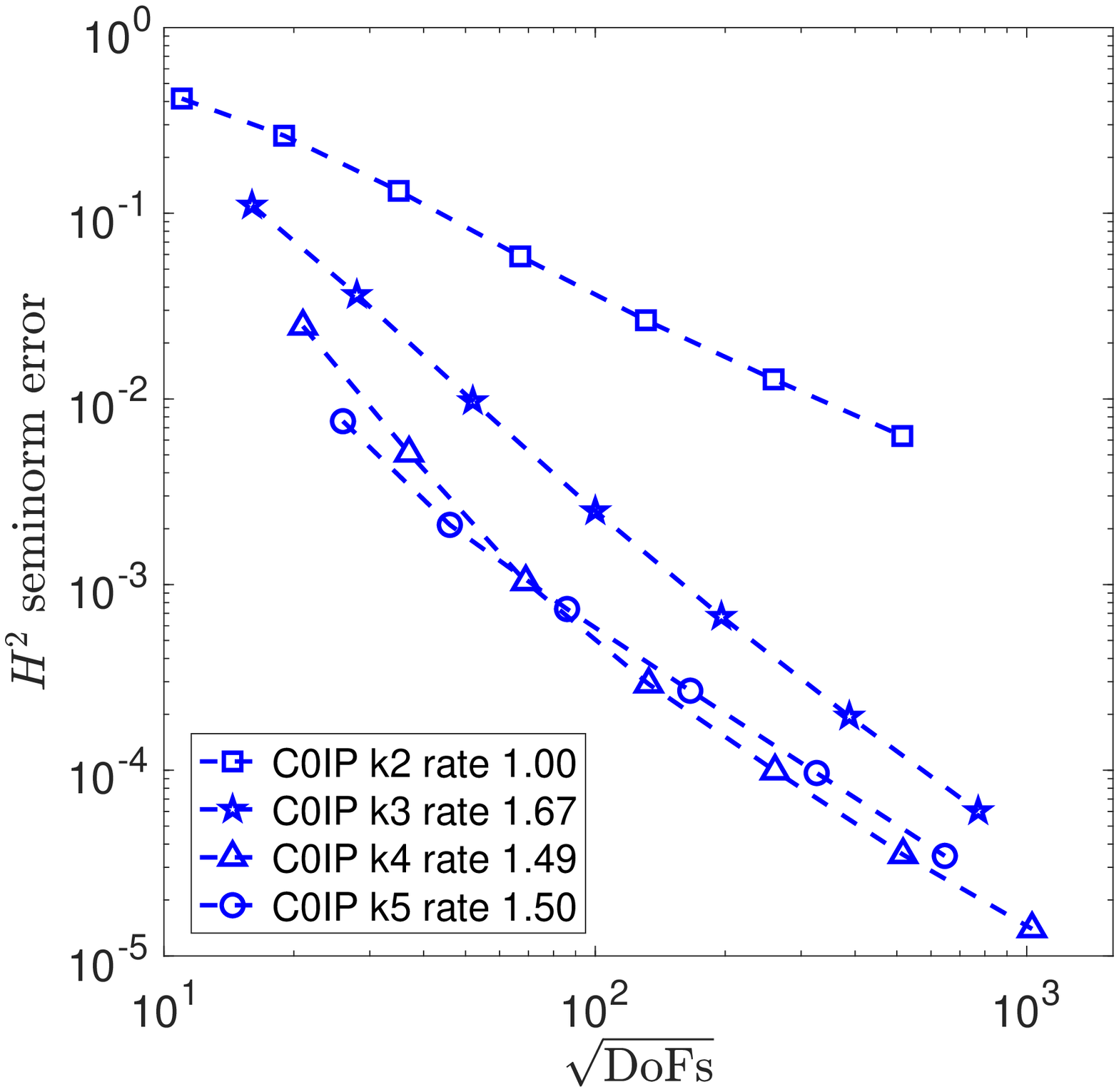}
	\caption{
		\label{ex2:h_refine_nonaligned_meshes}
		Convergence of $C^0$-HHO and $C^0$-IPDG methods in $H^2$-seminorms on triangular meshes with $k=0,1,2,3$.}
\end{figure}

\bibliographystyle{siam}
\bibliography{HHO_N}


\end{document}